\documentclass[a4paper,11pt]{article}
\usepackage{amsfonts}
\usepackage[T1]{fontenc}
\usepackage{microtype}
\usepackage{makeidx}
\usepackage{setspace}
\usepackage{authblk}
\usepackage{graphicx}
\usepackage[all]{xy}
\usepackage{amsmath}
\usepackage{amssymb}
\usepackage{booktabs}
\usepackage{braket}
\usepackage{array}
\usepackage{tabularx}
\usepackage{multirow}
\usepackage{indentfirst}
\usepackage{cite}
\usepackage{stmaryrd}
\usepackage{faktor}
\usepackage{titling}
\usepackage{tikz}
\usepackage{dsfont}
\usepackage{url}
\usepackage{hyperref}
\usepackage{amsthm}
\usepackage{tikz-cd}

\usetikzlibrary{matrix,arrows,decorations.pathmorphing}
    \renewcommand{\leq}{\leqslant}
    \renewcommand{\geq}{\geqslant}

\DeclareMathOperator{\sgn}{\mathrm{sgn}}

\DeclareMathOperator{\negg}{neg}
\DeclareMathOperator{\oneg}{oneg}
\DeclareMathOperator{\Neg}{\mathrm{Neg}}
\DeclareMathOperator{\eNeg}{\mathrm{eNeg}}
\DeclareMathOperator{\inv}{inv}
\DeclareMathOperator{\nsp}{nsp}

\DeclareMathOperator{\maj}{maj}

\DeclareMathOperator{\emaj}{emaj}
\DeclareMathOperator{\odes}{odes}

\DeclareMathOperator{\onegasc}{\mathrm{onegasc}}
\DeclareMathOperator{\negdes}{\mathrm{negdes}}
\DeclareMathOperator{\enmaj}{\mathrm{enmaj}}
\DeclareMathOperator{\nmaj}{\mathrm{nmaj}}
\DeclareMathOperator{\oposdes}{\mathrm{oposdes}}
\DeclareMathOperator{\des}{des}
\DeclareMathOperator{\pos}{pos}
\DeclareMathOperator{\W}{\mathcal{W}}
\DeclareMathOperator{\Z}{\mathbb{Z}}
\DeclareMathOperator{\RR}{\mathbb{R}}
\DeclareMathOperator{\rk}{\mathrm{rk}}
\DeclareMathOperator{\modue}{(\mathrm{mod}~2)}

\theoremstyle{plain}
\newtheorem{thm}{Theorem}[section]
\newtheorem{pro}[thm]{Proposition}
\newtheorem{lem}[thm]{Lemma}
\newtheorem{con}[thm]{Conjecture}
\newtheorem{cor}[thm]{Corollary}

\newtheorem{prob}{Problem}

\theoremstyle{definition}

\theoremstyle{remark}
\newtheorem{rmk}[thm]{Remark}
\newtheorem{exm}[thm]{Example}

\newcommand{\sm}{\sigma^{-1}}

\newcommand{\N}{\mathbb N}

\newcommand{\eqdef}{:=}

\newcommand{\PP}{\mathbb P}

\title{\bf{Wachs permutations, Bruhat order and weak order}}
\author{Francesco Brenti \thanks{Dipartimento di Matematica,
Universit\'{a} di Roma ``Tor Vergata'',
Via della Ricerca Scientifica, 1,
00133 Roma, Italy. \href{mailto:brenti@mat.uniroma2.it }{brenti@mat.uniroma2.it } } \ and Paolo Sentinelli\thanks{ Dipartimento di Matematica, Politecnico di Milano, Milan, Italy. \href{mailto:paolosentinelli@gmail.com}{paolosentinelli@gmail.com}}}
\date{}
\begin{document}
\maketitle















\begin{abstract}
We study the partial orders induced on Wachs and signed Wachs permutations by the Bruhat and weak orders of the symmetric and hyperoctahedral groups. We show that these orders are graded,
determine their rank function, characterize their ordering and covering
relations, and compute their characteristic polynomials, when
partially ordered by Bruhat order, and determine their structure explicitly when partially ordered by right weak order.
\end{abstract}

\section{Introduction}

Wachs permutations are a class of permutations first introduced (in the even case) in \cite{Wachs} to study the signed Eulerian numbers and the signed major index enumerator of the symmetric groups. This class was extended in \cite{BS} to the odd case and to signed permutations in order to study the enumerators of the odd and even major indices of classical Weyl groups twisted by their one-dimensional characters. In this work we study Wachs (and signed Wachs) permutations in their own right. More precisely, we study the partial orders induced on them by the Bruhat and weak orders of the symmetric  and hyperoctahedral groups. These orders are fundamental objects in algebraic combinatorics and have important connections to algebra and geometry  (see, e.g.,
\cite[Chaps.2 and 3]{BB}, \cite[Chap.2]{BL}, \cite[Chap.1]{Mac},
\cite[Chap.5]{Hum}, \cite[Chap.10]{Ful}, \cite[Chap.2]{Man}, and the references cited there). Many subsets of the symmetric and hyperoctahedral groups have been studied as posets under the Bruhat order such as, for example, quotients, descent classes, and generalized quotients \cite{BjWa}, complements of quotients \cite{Sen},
involutions \cite{Hul,Inc1,Inc2,Inc3,RiSp}, conjugation-invariant sets of involutions
\cite{Che,Han}, and twisted identities \cite{HultJAC}. In this paper we show that Wachs permutations possess many nice properties when partially ordered by Bruhat and weak orders. More precisely, we show that they form a graded poset, determine their rank function, characterize the ordering and covering relations, and compute the characteristic polynomial, when partially ordered by Bruhat order, and determine their structure explicitly when partially ordered by right weak order.

The organization of the paper is as follows. In the next section we collect some notation and results that are used in the sequel. In \S 3 we study the poset obtained by partially ordering Wachs permutations with respect to Bruhat order. We show that these posets are always graded (Theorem \ref{Dposet-odd}), characterize their order and covering relations (Theorem \ref{bruhat-odd} and Proposition \ref{cover-odd}), and compute
their characteristic polynomials (Corollary \ref{charpolyA}). In \S 4, using some of the results in \S 3, we obtain analogous results for the poset induced on signed Wachs permutations by Bruhat order. More precisely, we show that these posets are graded (Theorem \ref{mainB}), characterize their order and covering relations (Theorem \ref{Bbruhat-odd}, Proposition
\ref{cor Bruhat B}, Corollary \ref{coverB-odd}, and Corollary \ref{coverB-even}), and compute the characteristic polynomials (Corollary \ref{charpolyB}). In \S 5 we study the posets obtained by partially ordering Wachs and signed Wachs permutations under right weak order.
We show that they are always isomorphic to the direct product of a Boolean algebra with the weak order  on the whole group in rank $\left\lceil \frac{n}{2}\right\rceil$ (Theorems
\ref{weak-B} and \ref{weak-A}). Finally, in \S 6, we discuss some conjectures and open problems arising from the present work and the evidence that we have about them.

\section{Preliminaries}
In this section we recall some notation, definitions, and results that are used in the sequel.
As $\mathbb{N}$ we denote the set of
non-negative integers and as $\mathbb{P}$ the set of positive integers. If $n\in \mathbb{N}$,
$[n]:=\{1,2,...,n\}$; in particular $[0]=\varnothing$.  For $n\in \mathbb{P}$, in the polynomial ring $\mathbb{Z}[q]$ the $q$-analogue of $n$ is defined by $[n]_q:=\sum \limits_{i=0}^{n-1}q^i$ and the
$q$-factorial by
$[n]_q!:=\prod \limits_{i=1}^n[i]_q$.  For a set $A$ and $f:A \rightarrow \mathbb{N}$ we define
$$A(x,f):=\sum \limits_{w\in A}x^{f(w)}$$ and $f(u,v):=f(v)-f(u)$, for all $u,v\in A$.
The cardinality of a set $X$ will be denoted by $|X|$ and the power set of $X$ by $\mathcal{P}(X)$. Given a cartesian product $X\times Y$ of two sets $X$ and $Y$,  we indicate with $\pi_1$ and $\pi_2$ the projections on $X$ and $Y$ respectively.


Let $n\in \mathbb{N}$, $i \in \mathbb{Z}$, $q \in \mathbb{Q}$ and $J \subseteq [n]$; then we define
  $ J_e := \{j\in J:j\equiv 0\mod 2\}$,
  $J_o := \{j\in J:j\equiv 1\mod 2\}$,
  $ J+i := \{i+j:j\in J\} \cap [n]$, and
  $ qJ := \bigcup \limits_{j\in J} \{qj\} \subseteq \mathbb{Q}$.

We follow Chapter
3 of \cite{StaEC1} for notation and terminology concerning posets. We just recall some definitions. Given two posets $P$ and $Q$,
their \emph{ordinal product} $P \otimes Q$ is defined by ordering the set $P \times Q$ in the following manner: $(x,y) \leqslant (x',y')$ if and only if
$x=x'$ and $y\leqslant y'$ or $x<x'$, for all $(x,y),(x',y') \in P \times Q$. If $\mu$ is the M\"obius function of a graded poset $P$ with minimum $\hat{0}$, maximum $\hat{1}$ and rank function $\rho$, then the characteristic polynomial of $P$ in the indeterminate $x$ is defined by $$C_P(x):=\sum \limits_{z\in P}\mu(\hat{0},z)x^{\rho(z,\hat{1})}.$$

Next recall some basic results in the theory of
Coxeter groups which
will be useful in the sequel. The reader can consult  \cite{BB}
or \cite{Hum} for further details.
Let $(W,S)$ be a Coxeter system. The length of an element $z\in W$ with respect of the given presentation is denoted as
$\ell(z)$. If $J\subseteq S$, we let
\begin{eqnarray*}
W^J&:=&\{w\in W : \ell(ws)>\ell(w)~\forall~s\in J\},
\\ {^JW}&:=&\{w\in W : \ell(sw)>\ell(w)~\forall~s\in J\},
\\ D(w)&:=&\{s\in S : \ell(ws)<\ell(w)\},
\end{eqnarray*}
and, more generally, for any $A\subseteq W$ we let $A^J:=A\cap W^J$. The subgroup $W_J \subseteq W$ is the group generated by $J$.
In particular $W_S=W$ and $W_\varnothing =\{e\}$, being $e$ the identity of $W$.
We consider on $W$ the Bruhat order $\leqslant$ (see, e.g.,  \cite[Chapter
2]{BB} or \cite[Chapter 5]{Hum}) and on any subset
we consider the induced order. When the group $W$ is finite, there exists a unique
maximal element $w_0$ of maximal length. We recall this
characterizing property of the Bruhat order, known as the
\emph{lifting property} (see \cite[Proposition 2.2.7]{BB}):
\begin{pro} \label{sollevamento}
Let $v,w\in W$ be such that $v<w$ and $s\in D(w)\setminus D(v)$. Then $v\leqslant
ws$ and $vs\leqslant w$.
\end{pro}

For $J\subseteq S$, each element $w\in W$ has a unique expression
$w=w_Jw^J$, where $w^J\in W^J$ and $w_J\in W_J$ (see \cite[Proposition 2.4.4]{BB}). Often we consider the projection
$P^J:W \rightarrow W^J$ defined by  $P^J(w)=w^J$. This map is
order preserving (\cite[Proposition 2.5.1]{BB}).
 We let $T:=\{wsw^{-1}:w\in W,s\in S\}$. The following lemma will be useful in the next section; for a proof see \cite[Lemma 2.2.10]{BB}.
\begin{lem} \label{lemma xy}
  Suppose that $x < xt$ and $y < ty$, for $x, y \in W$, $t \in T$.
Then, $xy < xty$.
\end{lem} For $w\in W$ we set $T_L(w):=\{t\in T: \ell(tw)<\ell(w)\}$;
the \emph{right weak order}  $\leqslant_R$ on $W$ is the partial order whose cover relations are defined by letting $u \vartriangleleft_R v$ if and only if
$v^{-1}u\in S$ and $\ell(v)=\ell(u)+1$, for all $u,v\in W$ (see \cite[Chapter 3]{BB}). Then the following characterization holds (see \cite[Proposition 3.1.3]{BB}):
\begin{pro} \label{weak-order}
Let $(W,S)$ be a Coxeter system. Then $u \leqslant_R v$ if and only if $T_L(u) \subseteq T_L(v)$, for all $u,v\in W$.
\end{pro}

Now, for any $n\in \mathbb{P}$, let $S_n$ be the group of all bijections of the set $[n]$. It is well known that
it is a Coxeter group with set of generators $\{s_1,s_2,...,s_{n-1}\}$, being
 $s_i$ the simple inversion given, in one line notation, by
$12...(i+1)i...n$. Given a permutation $\sigma = \sigma(1)\sigma(2)...\sigma(n)\in S_n$, the action of $s_i$ on the right is
given by $\sigma s_i = \sigma(1)\sigma(2)...\sigma(i+1)\sigma(i)...\sigma(n)$, for all $i\in [n-1]$.
For $i,j \in [n]$, the action on the right of a transposition  $(i,j)$, $i\neq j$, then is given by $\sigma (i,j) = \sigma(1)\sigma(2)...\sigma(i-1)\sigma(j)...\sigma(j-1)\sigma(i)...\sigma(n)$, for all $\sigma \in S_n$.
As a Coxeter group, identifying $\{s_1,s_2,...,s_{n-1}\}$ with $[n-1]$, we have that
(see e.g. \cite[Propositions 1.5.3 and 1.5.2]{BB})
$D(\sigma)=\{i\in [n-1]:\sigma(i)>\sigma(i+1)\}$ and $\ell(\sigma)=\ell_A(\sigma)$,  where $\ell_A(\sigma):=\inv(\sigma)=|\{(i,j) \in [n]^2:i<j,\sigma(i)>\sigma(j)\}|$, for all $\sigma \in S_n$. Then, given $J \subseteq [n-1]$, $S^J_n=\{\sigma \in S_n:\sigma(i)<\sigma(i+1)~\forall ~i \in J\}$.

For $i \in [n]$ and $A \subseteq S_n$ define $A(i):=\{\sigma \in A:\pos(\sigma)=i\}$, being $\pos : S_n \rightarrow [n]$ the function defined by
$\pos (\sigma) \eqdef \sm(n)$, for all $\sigma \in S_n$. We find it convenient to define the following involution on $[n]$:
 $$i^*:=\left\{
     \begin{array}{ll}
       i-1, & \hbox{if $i\equiv 0 \mod 2$;} \\
       i+1, & \hbox{if $i\equiv 1 \mod 2$ and $i+1 \in [n]$;}\\
       n, & \hbox{otherwise,}
     \end{array}
   \right.
 $$ and the simple inversion $s^*_i:=(i,i^*)$, for all $i\in [n]$.
Given a permutation $\sigma \in S_n$, $k\in [n]$ and $i\in [k]$, define $\sigma_{i,k}$ as the $i$-th element in the increasing rearrangement of $\{\sigma(1),\sigma(2),...,\sigma(k)\}_<$ and
$\sm_{i,k}$ as the position of $\sigma(i)$ in $\{\sigma(1),\sigma(2),...,\sigma(k)\}_<$. So $\sigma_{j,k}=\sigma(i)$ if $j=\sm_{i,k}$. The following theorem, known as the ``tableau criterion'', characterizes the Bruhat
order on $S_n$ (see \cite[Theorem 2.6.3]{BB}).
\begin{thm} \label{tableau}
  Let $n\in \mathbb{P}$ and $\sigma, \tau \in S_n$. The following are equivalent:
  \begin{enumerate}
    \item $\sigma\leqslant \tau$;
    \item $\sigma_{i,k} \leqslant \tau_{i,k}$ for all $k\in D(\sigma)$ and $i\in [k]$;
    \item $\sigma_{i,k} \leqslant \tau_{i,k}$ for all $k\in [n-1]\setminus D(\tau)$ and $i\in [k]$.
  \end{enumerate}
\end{thm} We indicate by $w_n$ the maximum of the poset $(S_n,\leqslant)$, i.e., in one line notation, $w_n=n...321$.

The \emph{descent number} and the \emph{major index} are the functions $\des : S_n \rightarrow \mathbb{N}$ and
$\maj : S_n \rightarrow \mathbb{N}$ defined respectively by
$\des(\sigma)\eqdef |D(\sigma)|$ and
$\maj (\sigma) \eqdef \sum_{i \in D(\sigma)} i$. A famous result of McMahon asserts that $\ell_A$ and $\maj$ are equidistribuited over $S_n$ (see, e.g. \cite[Proposition 1.4.6]{StaEC1}), i.e. $S_n(x,\ell_A)=S_n(x,\maj)$, and this rank-generating function is known to be (see, e.g \cite[Corollary 1.3.10]{StaEC1})
\begin{equation}\label{poincare}
  S_n(x,\ell_A)=[n]_x!.
\end{equation} Following \cite{BS} we define functions from $S_n$ to $\N$ by letting
\[
\odes(\sigma)\eqdef |D(\sigma)_o|, \;\;\; \emaj (\sigma) \eqdef \sum_{i \in D(\sigma)_e} \frac{i}{2},
\]
for all $\sigma \in S_n$, and call these functions \emph{odd descent number} and \emph{even major index} respectively.

Following \cite{BS}, for $n \in \mathbb{P}$, we let
\[
\W(S_n) := \{ \sigma \in S_n : | \sm(i)-\sm(i^{\ast})| \leq 1 \; \mbox{if} \;i \in [n-1]  \}
\]
and call the elements of $\W(S_n)$
\textit{Wachs permutations}. It is not hard to see that, if $n$ is even,
$\W(S_n) = \{ \sigma \in S_n : | \sigma(i)-\sigma(i^{\ast})| \leq 1 \; \mbox{if} \;i \in [n-1]  \}$.

Let $[\pm n]:=\{-n,\ldots,-1,1,\ldots,n\}$. We denote by $B_n$ the group of bijective functions $\sigma : [\pm n] \rightarrow [\pm n]$ satisfying
 $-\sigma(i)=\sigma(-i)$, for all $i\in [n]$. We use the window notation. So, for example, the element
 $[-2,1]\in B_2$ represents the function $\sigma : [\pm 2] \rightarrow [\pm 2]$ such that
 $\sigma(1)=-2=-\sigma(-1)$ and $\sigma(2)=1=-\sigma(-2)$.
We let $\Neg(\sigma):=\{i\in [n]:\sigma(i)<0\}$, $\negg(\sigma)=|\Neg(\sigma)|$, $\nsp(\sigma):=|\{  (i,j) \in [n]^2 : \; i<j, \; \sigma(i)+\sigma(j)<0  \}|$, $s^B_{j} \eqdef (j,j+1)(-j,-j-1)$ for $j=1,...,
n-1$, $s_{0} \eqdef (1,-1)$, and $S_{B} \eqdef \{ s_{0},s^B_1,...,s^B_{n-1} \}$.
It is well
known that $(B_n,S_{B})$ is a Coxeter system and that, identifying $S_B$ with
$[0,n-1]$, the
following holds (see, e.g., \cite[\S 8.1]{BB}).

\begin{pro}
\label{CombB}
Let $\sigma \in B_n$. Then
$\ell_B(\sigma)=\ell_A(\sigma)+\negg(\sigma)+\nsp(\sigma)$, and
$D(\sigma)=\{i\in [0,n-1]:\sigma(i)>\sigma(i+1)\}$,
where $\sigma(0):=0$.
\end{pro}
We denote by $\leqslant$ the Bruhat order of $B_n$ and by $\sigma \mapsto \tilde{\sigma}$ the embedding $B_n \hookrightarrow S_{\pm n}$ (where $S_{\pm n}$ is the set of all bijections of $[\pm n]$) . The following result is \cite[Corollary 8.1.9]{BB}.
\begin{pro}
\label{prop Bruhat B}
We have that $\sigma \leqslant \tau$ in $B_n$ if and only if $\tilde{\sigma} \leqslant \tilde{\tau}$ in $S_{\pm n}$.
\end{pro}

The next result, which appears in \cite[Theorem 5.5]{incitti-Bruhat}, characterizes  the cover relations in the Bruhat order of $B_n$.
For $\sigma \in B_n$ and $i,j \in [\pm n]$, $i<j$, we say that $(i,j)$ is a {\em rise} for $\sigma$ if $\sigma(i) < \sigma(j)$ (i.e., if $(i,j)$ is not an inversion of
$\sigma$). Given a rise $(i,j)$ for $\sigma$ we then say, following
\cite{incitti-Bruhat}, that $(i,j)$ is  {\em central} if
$(0,0) \in [i,j] \times [\sigma(i), \sigma(j)]$, that $(i,j)$ is
{\em free}
if there is no $i<k<j$ such that $\sigma(i) < \sigma(k) < \sigma(j)$,
and that it is {\em symmetric} if $i=-j$.

\begin{thm} \label{incitti}
  Let $n\in \mathbb{P}$ and $\sigma, \tau \in B_n$. Then $\sigma \lhd \tau$ in Bruhat order if and only if either
 \begin{enumerate}
     \item $\tau=\sigma (i,j)(-i,-j)$ where $(i,j)$ is a non-central
     free rise of $\sigma$ or
     \item $\tau=\sigma (i,j)$ where $(i,j)$ is a central symmetric
     free rise of $\sigma$.
 \end{enumerate}
\end{thm}

Following \cite{BS}, for $n \in \mathbb{P}$, we let
\[
\W(B_n) := \{ \sigma \in B_n : | \sm(i)-\sm(i^{\ast})| \leq 1 \; \mbox{if} \;i \in [n-1]  \}
\]
and call the elements of $\W(B_n)$
\textit{signed Wachs permutations}.

Our aim in this work is to study the sets of Wachs and signed Wachs permutations under the Bruhat order and the weak order.

\section{Wachs permutations and Bruhat order} \label{section Deven}

Let $(W,S)$ be a Coxeter system. For any independent set  $I \subseteq S$ (i.e.  $st=ts$ for all $s,t\in I$)
we define a subgroup of $W$ by
$$
G_I:=\{ w\in W: I^w=I \},
$$
where $I^w:= \{ ws w^{-1} : s \in I \}$.
Note that $G_I$ is a subgroup of $W$, and that, since $I$ is
independent, $W_I \subseteq G_I$. Furthermore, $W_I$ is normal in $G_I$.

\begin{pro} \label{prop isomorfismo GI}
$P^I(G_I)$ is a subgroup of $G_I$, isomorphic to the quotient $G_I/W_I$. In particular we have the group isomorphisms
$$G_I \simeq P^I(G_I) \ltimes W_I \simeq S_2\wr_I P^I(G_I).$$
\end{pro}
\begin{proof}
 Let $w \in G_I$ and write $w= w^I w_I$ where $w^I \in W^I$ and $w_I \in W_I$. Let $s \in I$.
Then, by our hypothesis, $wsw^{-1}=w^I w_I s (w_I)^{-1} (w^I)^{-1}=w^I s (w^I)^{-1}$ so
$w^I \in G_I$. So
$P^I(G_I)\subseteq G_I$. Let $u,v \in G_I$ and write $u=u^I u_I$ and $v=v^I v_I$ where
$u^I, v^I \in W^I$ and $u_I, v_I \in W_I$, so $u^I,v^I \in P^I(G_I)$. Let $s \in I$. Then there is
$t \in I$ such that $v^I s = t v^I$. Furthermore, since $u^I, v^I \in W^I$, $\ell(t v^I)>\ell(v^I)$,
and $\ell(u^I t)>\ell(u^I)$. Therefore, by
Lemma \ref{lemma xy} we obtain that $u^I v^I s= u^I t v^I>u^I v^I$. Hence $u^I v^I \in W^I$.
But, since $v^I \in G_I$, there exists $\tilde{u}_I \in W_I$ such that $uv=u^I v^I \tilde{u}_I v_I$
so $(uv)^I=u^I v^I$. Hence $u^I v^I \in P^I(G_I)$. In particular, $(u^I)^{-1}=(u^{-1})^I$, so
$P^I(G_I)$ is a subgroup of $G_I$, and $P^I : G_I \rightarrow P^I(G_I)$ is a surjective homomorphism
whose kernel is $W_I$. The last statements follow by the definitions.
\end{proof}


Let $m \in \PP$, $W:= S_{2m}$. For any set $X$, we consider $\mathcal{P}(X)$ as an abelian group, the operation being the symmetric difference $+$,
i.e. $A+B:=\left(A\setminus B\right)\cup \left(B\setminus A\right)$, for all $A,B \in X$.  Then it is straightforward to see that if we take $I=\{s_i:i\equiv 1 \mod 2\} \subseteq S$ then
$\W(S_{2m})=G_I$. The group isomorphisms $P^I(\W(S_{2m})) \simeq S_m$ and $W_I \simeq \mathcal{P}([m])$ then imply
$$
\W(S_{2m}) \simeq S_m \ltimes \mathcal{P}([m])=S_2 \wr S_m.
$$
Therefore $\W(S_{2m})$ is isomorphic to the hyperoctahedral group.

We consider $S_n$ with the Bruhat order and the subset $\W(S_n) \subseteq S_n$ with the induced order. It is not difficult to see, using Theorem \ref{tableau}
that $P^I(\W(S_{2m})) \simeq S_m$ as posets for all $m>0$, where the set $S_m$ is ordered by the Bruhat order.
Let $m>0$. For $u \in S_m$ and $T\subseteq [m]$ let $\phi_{2m}^{-1}(u,T):=v$, where $v\in \W(S_{2m})$ is defined by
$$v(2i-1)=\left\{
    \begin{array}{ll}
      2u(i)-1, & \hbox{if $i\not \in T$,} \\
     2u(i), & \hbox{if $i\in T$,}
    \end{array}
  \right.
$$ and
$$v(2i)=\left\{
    \begin{array}{ll}
      2u(i), & \hbox{if $i\not \in T$,} \\
      2u(i)-1, & \hbox{if $i \in T$,}
    \end{array}
  \right.
$$ for all $i \in [m]$. The following result follows easily from our definitions and Theorem \ref{tableau},
and its proof is omitted.


\begin{pro} \label{Dposet}
	Let $m  \in \PP$. Then
	\begin{enumerate}
		\item $\phi_{2m}$ is a bijection;
		\item $\phi_{2m}: \W(S_{2m}) \rightarrow S_m \otimes \mathcal{P}([m])$ is order preserving;
		\item $\phi_{2m}^{-1}: S_m \times \mathcal{P}([m])\rightarrow \W(S_{2m})$ is order preserving.
	\end{enumerate} Moreover $\ell(v)=4\ell(\tau)+|T|$ if $\phi(v)=(\tau,T)$.
\end{pro}

Figure \ref{fig-D5} shows the Hasse diagram of $(\W(S_5),\leqslant)$. 
By the following example we see that $\phi_{2m}$ and $\phi_{2m}^{-1}$ are not poset isomorphisms
in general.

\begin{figure} \begin{center}\begin{tikzpicture}

\matrix (a) [matrix of math nodes, column sep=0.2cm, row sep=0.6cm]{
 & 4321 &  \\
3421 & & 4312   \\
& 3412 &  \\
 & 2143 &  \\
1243 & & 2134 \\
& 1234 & \\};

\foreach \i/\j in {1-2/2-1, 1-2/2-3,%
2-1/3-2,2-3/3-2, 3-2/4-2, 4-2/5-1,4-2/5-3, 5-1/6-2,5-3/6-2}
    \draw (a-\i) -- (a-\j);

\end{tikzpicture} \caption{Hasse diagram of $(\W(S_4),\leqslant)$.} \label{fig-D4} \end{center} \end{figure}

\begin{exm}
  Let $m=3$. Then
  $(123,\{1,2,3\}) \leqslant (132,\varnothing)$ in $S_m \otimes \mathcal{P}([m])$ but
  $\phi_{2m}^{-1}(123,\{1,2,3\})=214365 \nleqslant 125634=\phi_{2m}^{-1}(132,\varnothing)$ in $\W(S_{2m})$.

Let $m=2$, $u=2143$ and $v=3412$. Then $u,v \in \W(S_4)$, $u<v$ and $\phi_{2m}(u)=(12,\{1,2\}) \nleqslant (21,\varnothing)=\phi_{2m}(v)$
in $S_m \times \mathcal{P}([m])$.
\end{exm}


We consider now $\W(S_{2m+1})$. For any $m>0$ it is not difficult to see that the set $\W(S_{2m+1})$ is not a group. The previous general construction in Coxeter systems gives, for $W=S_{2m+1}$, the group $G_I \simeq \W(S_{2m})$ which, as a set, can be included in $\W(S_{2m+1})$. For $n>1$, define a function $\chi_n:\W(S_n) \rightarrow \W(S_{n-1})$ by

$$\chi_n(v)(i)=\left\{
         \begin{array}{ll}
           v(i), & \hbox{if $i<\pos(v)$;} \\
           v(i+1), & \hbox{if $i \geqslant \pos(v)$,}
         \end{array}
       \right.
$$
for all $i\in [n-1]$, $v\in \W(S_n)$. 
Note that the function $\chi_n$ is not necessarily order preserving. For example,
if $m=2$ and $u=21345$, $v=51234$ then $u \leqslant v$ but $\chi_n(u)=2134 \nleqslant 1234 = \chi_n(v)$.

Let $\phi_{2m+1} : \W(S_{2m+1})\rightarrow [m+1]\times S_m \times \mathcal{P}([m])$ be the function defined by
$$\phi_{2m+1}(v):=((\pos(v)+1)/2,\tau,T),$$ if $\chi_{2m+1}(v)=(\tau,T)$, for all $v\in \W(S_{2m+1})$. For example,  $\phi_{2m+1}(4312756)=(3,213,\{1\})$.
As in the even case, noting that $u\leqslant v$ implies $\pos(v)\leqslant \pos(u)$, we have the following result whose proof follows easily from our definitions and Proposition \ref{Dposet}.

\begin{pro} \label{Dposet2}
  Let $m>0$. Then
\begin{enumerate}
  \item $\phi_{2m+1}$ is a bijection;
  \item $\phi_{2m+1}: \W(S_{2m+1}) \rightarrow [m+1]^*\otimes S_m \otimes \mathcal{P}([m])$ is order preserving;
  \item \label{punto3} $\phi_{2m+1}^{-1}: [m+1]^*\times S_m \times \mathcal{P}([m])\rightarrow \W(S_{2m+1})$ is order preserving,
\end{enumerate}
where $i \leqslant j$ in $[m+1]^*$ if and only if $j \leqslant i$, for all $i,j \in [m+1]$. Moreover $\ell(v)=4\ell(\tau)+|T|+2(m-i+1)$ if $\phi_{2m+1}(v)=(i,\tau,T)$.
\end{pro}

\begin{figure} \begin{center}\begin{tikzpicture}

\matrix (a) [matrix of math nodes, column sep=0.3cm, row sep=0.6cm]{
       &       &       & 54321 &       &       & \\
       &       & 53421 &       & 54312 &       & \\
       &       & 43521 &       & 53412 &       & \\
       & 52143 &       & 34521 &       & 43512 & \\
 51243 &       & 52134 &       & 34512 &       & 43215\\
 21543 &       & 51234 &       & 34215 &       & 43125\\
       & 12543 &       & 21534 &       & 34125 & \\
       &       & 12534 &       & 21435 &       & \\
       &       & 12435 &       & 21345 &       & \\
       &       &       & 12345 &       &       & \\};

\foreach \i/\j in {1-4/2-3, 1-4/2-5,%
2-3/3-3, 2-3/3-5, 2-5/3-5,%
3-3/4-4, 3-3/4-6, 3-5/4-2, 3-5/4-6,%
4-2/5-1, 4-2/5-3, 4-4/5-5, 4-6/5-5, 4-6/5-7,%
5-1/6-1, 5-1/6-3, 5-3/6-3, 5-5/6-1, 5-5/6-5, 5-7/6-5, 5-7/6-7,%
6-1/7-2, 6-1/7-4, 6-3/7-4, 6-5/7-6, 6-7/7-6, %
7-2/8-3, 7-4/8-3, 7-4/8-5, 7-6/8-5,%
8-3/9-3, 8-5/9-3, 8-5/9-5,%
9-3/10-4, 9-5/10-4}
    \draw (a-\i) -- (a-\j);

\end{tikzpicture} \caption{Hasse diagram of $(\W(S_5),\leqslant)$.} \label{fig-D5} \end{center} \end{figure}

Let $n>0$ and $f_n:\W(S_n) \rightarrow S_{\left\lfloor\frac{n}{2}\right\rfloor}$ be the function defined by the assignment $(i,\tau,T) \mapsto \tau$ if $n$ is odd, and
$(\tau,T) \mapsto \tau$ if $n$ is even. By Proposition \ref{Dposet}, if $n$ is even then $f_n$ is order preserving.
Corollary \ref{cor-morph-odd} states that the function $f_n$ is order preserving also in the odd case.
We also define a function $\ell_{\W}: \W(S_n) \rightarrow \mathbb{N}$ by $$\ell_{\W}(v):=\ell(v)-\ell(f_n(v)),$$ for all $v\in \W(S_n)$.
For example,
$\ell_{\W}(342156)=5-1=4$ and $\ell_{\W}(3472156)=9-1=8$.
The minimum and the maximum of the poset $(\W(S_n),\leqslant)$ are respectively the identity of $S_n$, which corresponds to
$(e,\varnothing)$ in the even case and to $(m+1,e,\varnothing)$ in the odd one,
and the element of maximal length $w_n$ of $S_n$, which corresponds to $(w_m,[m])$ in the even case and to $(1,w_m,[m])$ in the odd one, where $m:=\left\lfloor \frac{n}{2} \right\rfloor$. Moreover $\ell_{\W}(e)=0$ and
\begin{equation}
\label{lDw}
  \ell_{\W}(w_n)= {{n} \choose {2} }- {{\lfloor \frac{n}{2}  \rfloor}\choose{2}}.
\end{equation}

Note that if $n$ is even then $\ell_{\W}(w_n)=\frac{(3n-2)n}{8}=\frac{(3m-1)m}{2}$.
So $\{\ell_{\W}(w_{2m})\}_{m\in \mathbb{P}}$ is the sequence of pentagonal numbers (see A000326 in OEIS).
If $n$ is odd then $\ell_{\W}(w_n)=\frac{3(n^2-1)}{8}=\frac{3m(m+1)}{2}$;
so $\{\ell_{\W}(w_{2m+1})\}_{m\in \mathbb{P}}$ is the sequence of triangular matchstick numbers
(see A045943 in OEIS).

%

Let $n>0$ and $m:= \left\lfloor \frac{n}{2} \right\rfloor$; for any $i,j \in [m]$, $i<j$, we define an involution $w^A_{i,j}: \W(S_n) \rightarrow \W(S_n)$ by

$$w^A_{i,j}(v)=\left\{
                 \begin{array}{ll}
                   (\tau(i,j),T+\{i,j\}), & \hbox{if $n$ is even and $v=(\tau,T)$;} \\
                   (k,\tau(i,j),T+\{i,j\}), & \hbox{if $n$ is odd and $v=(k,\tau,T)$,}
                 \end{array}
               \right.
$$ for all $v\in \W(S_n)$. For example, if $v=4312765 \in \W(S_7)$, then
$w^A_{2,3}(v)=4356721$.


\begin{lem} \label{lemma-tij}
  Let $m>0$, $v=(\tau,T)\in \W(S_{2m})$, $i,j\not \in T$ and $\tau (i,j)\vartriangleleft \tau$. Then $w^A_{i,j}(v)<v$
  and $\ell_{\W}(v)-\ell_{\W}(w^A_{i,j}(v))=1$. If $u=(\sigma,S)\in \W(S_{2m})$, $u<v$ and $\sigma\leqslant \tau(i,j)$ then $u\leqslant w^A_{i,j}(v)$.
\end{lem}
\begin{proof}
We have that $\hat{v}:=w^A_{i,j}(v)=(\tau(i,j),T\cup \{i,j\})$. By the tableau criterion it is easy to deduce that $\hat{v}<v$ and the equality
$\ell_{\W}(\hat{v},v)=1$ follows easily from Proposition \ref{Dposet} and the definition of $\ell_{\W}$.
We now show that $u \leq \hat{v}$. We use Theorem \ref{tableau}. Note first that, since $u$ and $\hat{v}$ are Wachs permutations and $\sigma \leq \tau (i,j)$, $u_{l,h} \leq \hat{v}_{l,h}$ for all $1 \leq l \leq h$ if
$h$ is even. Furthermore, since $u<v$, $u_{l,h} \leq \hat{v}_{l,h}$ for all $l \in [h]$ if $h \leq 2i-2$ or
$h \geq 2j$. So let
$k\in [m]$ be such that $k \not \in T$ and $i< k < j$. We wish to show that $u_{l,2k-1} \leqslant \hat{v}_{l,2k-1},$ for all $1\leqslant l\leqslant 2k-1$.

Let $a:=v(2i-1)$ and $b:=v(2j-1)$. So $a+1=v(2i)$ and $a>b+1=v(2j)$ as well as $\hat{v}(2i-1)=b+1$, $\hat{v}(2i)=b$, $\hat{v}(2j-1)=a+1$
and $\hat{v}(2j)=a$. Let $c:=v(2k-1)$ and $d:=u(2k-1)$. Note that, since $\tau(i,j) \vartriangleleft \tau$, $c\not \in [b,a+1]$
and that $c \equiv 1 \modue$ since $k\not \in T$.
Let $r,s \in [k]$ be such that $v_{2r-1,2k-1}=c$ and $u_{2s-1,2k-1}=d$. We have two cases to consider according as to whether $c<b$ or $c>a+1$.

Say $c>a+1$.

Let $p,q\in [k]$ be such that $v_{2q-1,2k-1}=a$ and $\hat{v}_{2p-1,2k-1}=b$, so $p\leq q$, $v_{2q,2k-1}=a+1$ and $\hat{v}_{2p,2k-1}=b+1$.
Then $2r-1>2q$ since $c>a+1$. Note that

\begin{equation} \label{11} v_{l,2k-1}=\hat{v}_{l,2k-1}
\end{equation}
if $1\leqslant l \leqslant 2p-2$ or $2q+1\leqslant l \leqslant 2k-1$
and that

\begin{equation} \label{22} v_{l,2k-1}=\hat{v}_{l+2,2k-1}
\end{equation}
if $2p-1\leqslant l \leqslant 2q-2$.
Moreover, since $u\leqslant v$ by our hypothesis,

\begin{equation} \label{33} u_{l,2k-1}\leqslant v_{l,2k-1},
\end{equation} for all $1\leqslant l \leqslant 2k-1$.

Note that, by \eqref{11} and \eqref{33}, we have
$$u_{l,2k-1}\leqslant v_{l,2k-1}= \hat{v}_{l,2k-1}$$ if $1\leqslant l \leqslant 2p-2$ or  $2q+1\leqslant l \leqslant 2k-1$.
So assume $2p-1\leqslant l \leqslant 2q$.

If $x,y \in \mathbb{P}$, $x,y \equiv 1 \modue$, then we find it convenient to write $\{x,x+1\}\leqslant \{y,y+1\}$ if $x\leqslant y$ and similarly for $<$.

We have two cases to consider.

Let $s>q$. Then, since $\sigma \leqslant \tau (i,j)$, we have that
$$\{u_{2h-1,2k-1},u_{2h,2k-1}\} \leqslant \{\hat{v}_{2h-1,2k-1},\hat{v}_{2h,2k-1}\}$$  for all $p\leqslant h \leqslant q$, so $u_{l,2k-1}\leqslant \hat{v}_{l,2k-1}$ for all $2p-1\leqslant l \leqslant 2q$.

Let $s\leqslant q$.
If $p \leqslant h \leqslant s-1$ then, since $\sigma \leqslant \tau (i,j)$, we have that
$$
\{u_{2h-1,2k-1},u_{2h,2k-1}\} \leqslant \{\hat{v}_{2h-1,2k-1},\hat{v}_{2h,2k-1}\}
$$
so $u_{l,2k-1}\leqslant \hat{v}_{l,2k-1}$ for all $2p-1 \leq l \leq 2s-2$.

We may therefore assume that $\max\{2p-1,2s-1\} \leqslant l\leqslant 2q$. Since $\sigma \leqslant \tau (i,j)$, we have that
$$
u_{l,2k-2} \leqslant \hat{v}_{l,2k-2}
$$
for $1\leqslant l \leqslant 2k-2$. But
$$
u_{l,2k-1} = u_{l-1,2k-2}
$$
if $2s\leqslant l \leqslant 2k-1$ since
$\{u(1),...,u(2k-2)\}=\{u(1),...,u(2k-1)\}\setminus \{d\}$, and similarly
$$
\hat{v}_{l,2k-1} = \hat{v}_{l,2k-2}
$$
if $1\leqslant l \leqslant 2r-2$. Therefore
$$
u_{l,2k-1} < u_{l+1,2k-1}= u_{l,2k-2} \leqslant \hat{v}_{l,2k-2} =\hat{v}_{l,2k-1},
$$
if $2s-1\leqslant l\leqslant 2r-2$, so $u_{l,2k-1}\leqslant \hat{v}_{l,2k-1}$ if $2s-1 \leqslant l \leqslant 2q$, because $2q\leqslant 2r-2$.

The case $c<b$ is similar (and slightly simpler) except that one uses $2k$ rather than $2k-2$ and
obtains that
$$
u_{l,2k-1} = u_{l,2k} \leq \hat{v}_{l,2k} = \hat{v}_{l-1,2k-1} < \hat{v}_{l,2k-1}
$$
if $2p \leqslant l \leqslant 2s-1$. We leave the details to the interested reader.
\end{proof}

%




The following lemma can be proved using Lemma \ref{lemma-tij} and the definitions.
\begin{lem} \label{lemma-tij-odd}
  Let $m>0$, $v=(k,\sigma,S)\in \W(S_{2m+1})$, $i,j,k\not \in S$ and $\sigma(i,j)\vartriangleleft \sigma$. Then $w^A_{i,j}(v)<v$, $v(k,k+2)<v$
  and $\ell_{\W}(w^A_{i,j}(v))=\ell_{\W}(v(k,k+2))=\ell_{\W}(v)-1$.
\end{lem}

The next result is the crucial technical tool for the proof of
our main theorem, and states that elements of $\W(S_{2m+1})$
are ``join-irreducible'' in a certain sense.

\begin{pro} \label{lemma phin}
  Let $m>0$, $u,v\in \W(S_{2m+1})$, $u<v$ and $i:=\pos(v)$. If $u<v$ and $\pos(u,v)>0$ then $u\leqslant z$ where $z\in \W(S_{2m+1})$ is defined by
$$z=\left\{
  \begin{array}{ll}
    v(i,i+2), & \hbox{if $v<v(i+1,i+2)$;} \\
    v(i+1,i+2), & \hbox{otherwise.}
  \end{array}
\right.$$

\end{pro}
\begin{proof}
Let $h:=(i+1)/2$ and consider first the case $i+1 \not \in D(v)$.
Let $a:=v(i+1)$; then $a+1=v(i+2)=z(i)$, $a=z(i+1)$ and $z(i+2)=2m+1$. Define $c:=u(i)$, $r\in [m+1]$ be such that $u(2r-1)=2m+1$,
$p\in [h]$ be such that $z_{2p-1,i}=a+1$ and $q\in [h]$ be such that $u_{2q-1,i}=c$. Notice that $u(i+1)\in \{u(i)+1,u(i)-1\}$ since $\pos(u,v)>0$.
Assume that $q\leqslant p$. Note that
$$ u_{l,i} \leqslant v_{l,i}=z_{l,i}$$ if $1\leqslant l \leqslant 2p-2$, while
$$u_{l,i}=u_{l-1,i-1}\leqslant v_{l-1,i-1}=z_{l-1,i-1}=z_{l,i}$$ if $2p\leqslant l$, and
$$u_{2p-1,i}=\left\{
    \begin{array}{ll}
      u_{2p-2,i-1} \leqslant v_{2p-2,i-1}=z_{2p-2,i-1}=z_{2p-2,i}<z_{2p-1,i}, & \hbox{if $q<p$,} \\
      u_{2p-1,i+1} \leqslant v_{2p-1,i+1}=a<a+1=z_{2p-1,i}, & \hbox{if $q=p$ and $u(i)<u(i+1)$,} \\
      u_{2p-1,i+1}+1 \leqslant v_{2p-1,i+1}+1=a+1=z_{2p-1,i}, & \hbox{if $q=p$ and $u(i)>u(i+1)$,}
    \end{array}
  \right.
$$ since $u\leqslant v$. Moreover $u_{l,i+1}=u_{l-1,i}\leqslant z_{l-1,i}=z_{l,i+1}$ if $l \geqslant 2p+1$ while $u_{l,i+1}\leqslant v_{l,i+1} = z_{l,i+1}$, if $l\leqslant 2p-2$.
Let $l\in \{2p-1,2p\}$. There are some cases to be considered.
\begin{enumerate}
  \item $q<p$ and $u(i)>u(i+1)$: in this case we have that
$u_{2p,i+1}=u_{2p-1,i}\leqslant z_{2p-1,i}=z_{2p,i+1}$ and so
$u_{2p-1,i+1}\leqslant u_{2p,i+1}-1\leqslant z_{2p,i+1}-1=z_{2p-1,i+1}$.

  \item $q<p$ and $u(i)<u(i+1)$: in this case we have that $u_{2p-1,i+1}=u_{2p-1,i}-1\leqslant z_{2p-1,i}-1=z_{2p-1,i+1}$ and so
$u_{2p,i+1}=u_{2p-1,i+1}+1\leqslant z_{2p-1,i+1}+1=z_{2p,i+1}$.

  \item $q=p$ and $u(i)>u(i+1)$: in this case $u_{2p,i+1}=u_{2p-1,i}\leqslant z_{2p-1,i}=z_{2p,i+1}$ and $u_{2p-1,i+1}=u_{2p-1,i}-1\leqslant z_{2p-1,i}-1=z_{2p-1,i+1}$.

  \item $q=p$ and $u(i)<u(i+1)$: in this case $u_{2p-1,i+1}\leqslant v_{2p-1,i+1}=a=z_{2p-1,i+1}$
and $u_{2p,i+1}=u_{2p-1,i+1}+1\leqslant z_{2p-1,i+1}+1=z_{2p,i+1}$.
\end{enumerate}

%


So we have proved that $u\leqslant z$ whenever $q \leqslant p$.
Consider the case $q>p$. If $l\leqslant 2p-2$ or $l\geqslant 2q$ the result follows as above.
Let $2p \leqslant l \leqslant 2q-2$. Then $u_{l,i}=u_{l,i+1}\leqslant v_{l,i+1}=z_{l,i}$. Moreover
$$u_{2p-1,i}=u_{2p-1,i+1}\leqslant v_{2p-1,i+1}=a<a+1=z_{2p-1,i}.$$
Let $u(i)<u(i+1)$. Then
$$u_{2q-1,i}=u_{2q-1,i+1}\leqslant v_{2q-1,i+1}=z_{2q-1,i}.$$ If $u(i)>u(i+1)$ we have that
$u_{2q-1,i}=u_{2q-1,i+1}+1\leqslant v_{2q-1,i+1}+1=z_{2q-1,i}$ and then we have proved that $u \leqslant z$ also in case $q>p$.

Let's consider the case $s_{i+1} \in D(v)$. Define $a:=v(i+2)$ and $v_{2p-1,i+1}:=a+1$. If $l\neq 2p-1$ we have that $u_{l,i+1} \leqslant v_{l,i+1}=z_{l,i+1}$.
Let $l=2p-1$. Since $u_{2p-1,i+1} \equiv 1 \modue$, $a \equiv 1 \modue$ and $u_{2p-1,i+1}\leqslant a+1$, we conclude that $u_{2p-1,i+1}\leqslant a=z_{2p-1,i+1}$.

\end{proof}

\begin{cor} \label{cor-morph-odd}
  Let $n \in \mathbb{P}$. Then $f_n : \W(S_n) \rightarrow S_{\lfloor\frac{n}{2}\rfloor}$ is order preserving.
\end{cor}
\begin{proof}
 If $n$ is even the result was already observed. If $n$ is odd let $u,v\in \W(S_n)$ and $u\leqslant v$.
We prove the result by induction on $\pos(u,v)$. If $\pos(u,v)=0$ by the tableau criterion we conclude that $\chi_n(u)\leqslant \chi_n(v)$ and then the result follows by the even case. Let $\pos(u,v)>0$ and $i:=\pos(v)$.
In this case, by Proposition \ref{lemma phin},
  $u\leqslant v(i,i+2)<v$ or $u\leqslant v(i+1,i+2)(i,i+2)<v(i+1,i+2)<v$. Hence, by the inductive hypothesis, $f_n(u)\leqslant f_n(v(i,i+2))=f_n(v)$ or
$f_n(u)\leqslant f_n(v(i+1,i+2)(i,i+2))=f_n(v(i+1,i+2))=f_n(v)$.
\end{proof}

We can now prove the main result of this section.
\begin{thm} \label{Dposet-odd}
  Let $n>0$. Then $(\W(S_n),\leqslant)$ is graded, of rank
  ${{n} \choose {2} }- {{\lfloor \frac{n}{2}  \rfloor}\choose{2}}$, and its rank function is $\ell_{\W}$.
\end{thm}
\begin{proof}

Assume first $n=2m$, $m>1$. Let $u,v\in \W(S_{2m})$  with $u<v$, $\phi(u)=(\sigma,S)$ and $\phi(v)=(\tau,T)$. We prove that, if $\ell_{\W}(u,v) \neq 1$, then there exists $z\in \W(S_{2m})$ such that $u<z<v$ and $\ell_{\W}(z,v)=1$. Note
  that, by Proposition \ref{Dposet}, $\sigma \leq \tau$. We have two cases to consider.

 \begin{enumerate}
   \item $\ell(\sigma,\tau)=0$: then $\sigma=\tau$ so, by Proposition \ref{Dposet}, $S \subseteq T$.
   Therefore there exists $i \in T\setminus S$. So $s_{2i-1} \in D(v) \setminus D(u)$ hence, by the
  Lifting Property, $u < v s_{2i-1}<v$, $v s_{2i-1} \in \W(S_{2m})$
 and $\ell_{\W}(vs_i,v)=1$.

   \item $\ell(\sigma,\tau)>0$: in this case, since $\sigma < \tau$, there exist $1 \leq x < y \leq m$
such that $\sigma \leqslant \tau (x,y) \lhd \tau$.  Let $i:=2x-1$ and $j:=2y-1$. There are three cases to consider.
 \begin{enumerate}
\item $s_i\in D(v)$: let $z:=vs_i$ and $w:=(\tau (x,y),T)$. We want to prove that $u<z$. Let $a+1:=v(i)=v_{2p-1,i}$, $b:=w(i)=v(j)=w_{2q-1,i}$ and $c:=u(i)=u_{2s-1,i}$.
   In particular $q \leq p$.
  We only have to prove that $u_{l,i}\leqslant (vs_i)_{l,i}$ for all $l\leqslant i$. Since $(vs_i)_{l,i}=v_{l,i}$ if $l \neq 2p-1$, it is enough to
    prove that $u_{2p-1,i}\leqslant (vs_i)_{2p-1,i}$.
      Note that $u_{l,i+1} \leq v_{l,i+1} = (v s_i)_{l,i+1}$ for all $1 \leq l \leq i+1$.

If $s\leqslant p$ then
$$
(vs_i)_{2p-1,i} > (vs_i)_{2p-2,i} = w_{2p-1,i} = w_{2p,i+1} \geqslant u_{2p,i+1}\geqslant u_{2p-1,i}
$$
because $b<a$,  and $\sigma \leq \tau (x,y)$.

 If $s>p$ then $u_{2p-1,i}=u_{2p-1,i+1}\leqslant (vs_i)_{2p-1,i+1}=(vs_i)_{2p-1,i}$.
    Since $\ell_{\W}(v s_i, v)=1$ the result follows.

\item $s_i \not \in D(v)$ and $s_j\in D(v)$: we then claim that  $u<vs_j$.
     Let $b+1=v(j)=v_{2p-1,j}$; then $b=(vs_j)(j)=(vs_j)_{2p-1,j}$. As in the case above it is sufficient to prove that $u_{2p-1,j}\leqslant (vs_j)_{2p-1,j}$.
         If $u(j)>u_{2p-1,j}$ then $u_{2p-1,j}=u_{2p-1,j+1}=u_{2p,j+1}-1 \leq v_{2p,j+1}-1 =b$.
       Let $u(j)\leqslant  u_{2p-1,j}$; since $\sigma \leqslant \tau (x,y)$ we have $\{u_{2p-1,j-1},u_{2p,j-1}\}\leqslant \{w_{2p-1,j-1},w_{2p,j-1}\}$. Then, $u_{2p-1,j}<u_{2p,j}=u_{2p-1,j-1}\leqslant w_{2p-1,j-1}=b$, since $a > b+1$.

\item $s_i,s_j \not \in D(v)$: in this case, by Lemma \ref{lemma-tij}, $u<w^A_{x,y}(v)<v$ and $w^A_{x,y}(v)\in \W(S_{2m})$, $\ell_{\W}(w^A_{x,y}(v),v)=1$.

\end{enumerate}
\end{enumerate}

Now assume $n=2m+1$, $m>0$.
Let $u \leqslant v$, $i:=\pos(v)$, and $\ell_{\W}(u,v)>1$. We prove that there exists $z\in \W(S_{2m+1})$ such that $u<z<v$ and $\ell_{\W}(z,v)=1$.
If $\pos(u,v)=0$ then the result follows by the previous point.
If $\pos(u,v)>0$ we have, by Proposition \ref{lemma phin}, $u<v(i,i+2)<v$ if $s_{i+1}\not \in D(v)$ and $u< vs_{i+1}<v$ otherwise. 
\end{proof}

\begin{rmk}
  In general $\W(S_n) \cap S_n^J$ is not graded; one can see this by considering $J=\{s_1\}$ and the interval $[124365,561234]$ in $\W(S_6)\cap S_6^J$.
\end{rmk}

We can now compute the rank-generating function of $(\W(S_n),\leqslant)$.

\begin{cor} \label{poincareD} Let $m>0$. Then
$$\W(S_{2m})(x,\ell_{\W})=(1+x)^m[m]_{x^3}!,$$
  $$\W(S_{2m+1})(x,\ell_{\W})=(1+x)^m[m+1]_{x^2}[m]_{x^3}!.$$
  Moreover $$\W(S_{2m})(x,\ell_{\W})=\W(S_{2m})(x,3\emaj+\odes)$$
  and $$\W(S_{2m+1})(x,\ell_{\W})=\W(S_{2m+1})(x,(3\emaj+\odes) \circ \chi_{2m+1}+\pos).$$
\end{cor}
\begin{proof}
  The result follows by \eqref{poincare}, Theorem \ref{Dposet} and the definition of $\ell_{\W}$.
  In fact, $\W(S_{2m+1})(x,\ell_{\W})=[m+1]_{x^2}\W(S_{2m})(x,\ell_{\W})$ and $\ell(v)=\odes(v)+4\ell(f_{2m}(v))$, for all $v \in \W(S_{2m})$.
\end{proof}

From Proposition \ref{poincareD} we find that the polynomials $\W(S_n)(x,\ell_{\W})$ are reciprocal, i.e. $x^{\ell_{\W}(w_n)}$  $\W(S_n)(x^{-1},\ell_{\W})=\W(S_n)(x,\ell_{\W})$.
In fact the poset $(\W(S_n),\leqslant)$ is self-dual, for all $n\in \mathbb{P}$, since the map $v \mapsto vw_n$ is an antiautomorphism
of $(\W(S_n),\leqslant)$ such that $\ell_{\W}(vw_n)=\ell_{\W}(w_n)-\ell_{\W}(v)$ (see \cite[Propositions 2.3.2 and 2.3.4]{BB}).



From the combinatorial description of the rank function of $\W(S_{2m+1})$ we can deduce a description of its cover relations.

\begin{pro} \label{cover-odd}
  Let $m\in \mathbb{P}$, $u,v\in \W(S_{2m+1})$, $u=(i,\sigma,S)$ and $v=(j,\tau,T)$. Then $u\vartriangleleft v$ if and only if
  either
\begin{enumerate}
  \item $i=j$, $\sigma=\tau$ and $S\vartriangleleft T$, or
  \item $j=i-1$, $\sigma=\tau$, $i-1\in S$ and $T=S\setminus \{i-1\}$, or
  \item $i=j$, $\sigma \vartriangleleft \tau$, $T\cap \{a,b\}= \varnothing$ and $S=T\cup \{a,b\}$,
\end{enumerate} where $(a,b):=\tau^{-1}\sigma$.
In particular, if $u,v\in \W(S_{2m})$, $u=(\sigma,S)$ and $v=(\tau,T)$, then $u\vartriangleleft v$ if and only if
  either $\sigma=\tau$ and $S\vartriangleleft T$, or
  $\sigma \vartriangleleft \tau$, $T\cap \{a,b\}= \varnothing$ and $S=T\cup \{a,b\}$.
\end{pro}
\begin{proof}
  If point $1$ or $2$ hold then $\ell_{\W}(v)-\ell_{\W}(u)=1$ and by Theorem \ref{Dposet2} there follows that $u\leqslant v$.
  If point $3$ holds then the result follows by Lemma \ref{lemma-tij-odd}.

Conversely let $u\vartriangleleft v$. Then $\sigma\leqslant \tau$ by Corollary \ref{cor-morph-odd} and $\ell_{\W}(u,v)=1$.
  If $i-j\geqslant 2$ then, by Proposition \ref{lemma phin} there is $z\in \W(S_{2m+1})$ such that
  $u\leqslant z<v$ and $\ell_{\W}(z,v)=1$; so $z=u$, which is a contradiction since $k-j\leqslant 1$, being $z=(k,\rho,R)$. Hence  assume $i-j\leqslant 1$.

If $\sigma < \tau$ then there exists a reflection $(r,s)$ such that $\sigma \leqslant \tau (r,s)\vartriangleleft \tau$.
    If $i-j=1$ then by Proposition \ref{lemma phin} $u=v(2j-1,2j+1)$ and $\sigma=\tau$, which is a contradiction. Therefore $i=j$. As in the proof of Theorem \ref{Dposet-odd},
  if $r\in T$ then $(\sigma,S)\leqslant (\tau,T \setminus\{r\})$; if $r\not \in T$ and $s\in T$ then $(\sigma,S)\leqslant (\tau,T \setminus\{s\})$.
   Moreover $\ell_{\W}((\tau,T \setminus\{r\}),(\tau,T))=1$ and $\ell_{\W}((\tau,T \setminus\{s\}),(\tau,T))=1$ in these cases.
  Then $(\sigma,S)=(\tau,T \setminus\{r\})$ or $(\sigma,S)=(\tau,T \setminus\{s\})$, a contradiction.
   Therefore $r,s\not \in T$ and, as in the proof of Theorem \ref{Dposet-odd}, $(\sigma,S)\leqslant (\tau(r,s),T\cup \{r,s\})\vartriangleleft (\tau,T)$, and
$\ell_{\W}((\tau(r,s),T\cup \{r,s\}),(\tau,T))=1$. So
we conclude that $S=T\cup \{r,s\}$ and $\sigma=\tau(r,s)\vartriangleleft \tau$.

Assume now $\sigma=\tau$. If $i=j$ then $S\vartriangleleft T$ (else $\ell_{\W}(u,v)=\ell(u,v)\geqslant 2$).
If $i-j=1$ then by Proposition \ref{lemma phin} we have that $u=v(2j-1,2j+1)$ and $v<vs_{2j}$ and the result follows.

The second statement follows from the first one by observing that $\W(S_{2m})$ is isomorphic to the interval $[e,2m...321(2m+1)]$ in $\W(S_{2m+1})$.
\end{proof} For example the permutation $782156934\in \W(S_9)$ covers the Wachs permutations $781256934$,
$782156439$ and $652187934$, which correspond respectively to cases 1, 2 and 3 in Proposition \ref{cover-odd}.

From Proposition \ref{cover-odd} we can now prove the second main result of this section, namely a characterization of the Bruhat order relation on Wachs permutations.

\begin{thm} \label{bruhat-odd}
  Let $m>0$ and $u,v\in \W(S_{2m+1})$, $u=(i,\sigma,S)$, $v=(j,\tau, T)$. Then
$u\leqslant v$ if and only if $$\mbox{$\sigma \leqslant \tau$, $\,$ $S(u,v) \subseteq T(u,v)$, $\,$ and $\,$ $j\leqslant i$,}$$
where, for $X\subseteq [m+1]$, $X(u,v):=X\cap ([j-1] \cup [i,m]) \cap F(\sigma,\tau)$, being $F(\sigma,\tau):=\{k\in [m]:\sigma(k)=\tau(k)\}$.
Moreover $\ell_{\W}(u)=3\ell(\sigma)+|S|+2(m-i+1)$.
In particular, if $u,v\in \W(S_{2m})$, $u=(\sigma,S)$, $v=(\tau, T)$, then
$u\leqslant v$ if and only if $\sigma \leqslant \tau$ and $S\cap F(\sigma,\tau) \subseteq T$, and $\ell_{\W}(u)=3\ell(\sigma)+|S|$.
\end{thm}
\begin{proof}
  Let $u\leqslant v$. We may assume $u\vartriangleleft v$. There are three cases to consider.
\begin{enumerate}
  \item $i=j$, $\sigma = \tau$ and $S\vartriangleleft T$: in this case $F(\sigma,\tau)=[m]$ so the result follows.
  \item $j=i-1$, $\sigma = \tau$ and $S=T\cup \{i-1\}$, with $i-1\not \in T$: in this case $F(\sigma,\tau)=[m]$ and
   $S\cap ([i-2]\cup [i,m]) \subseteq T$.
  \item $i=j$, $\sigma \vartriangleleft \tau$, $T\cap \{a,b\}=\varnothing $ and $S=T\cup \{a,b\}$, where $(a,b)=\tau \sm$:
     we have $F(\sigma,\tau)=[m]\setminus \{a,b\}$ and then $S\cap F(\sigma,\tau)\subseteq T$.
\end{enumerate} Now let $\sigma \leqslant \tau$, $j<i$ and $S\cap ([j-1] \cup [i,m]) \cap F(\sigma,\tau) \subseteq T$.
In $S_m$ there exists a saturated chain $\sigma = \sigma_0 \vartriangleleft \sigma_1 \vartriangleleft ... \vartriangleleft \sigma_n=\tau$ with $n=\ell(\sigma,\tau)$.
Define $(a_i,b_i):=\sm_i\sigma_{i-1}$ for all $i\in [n]$. We have the following chain in $(\W(S_{2m+1}),\leqslant)$:
\begin{eqnarray*}
  (i,\sigma, S) &\leqslant& (i,\sigma, S\cup [j,i-1]) \\
  &\vartriangleleft& (i-1,\sigma, (S\cup [j,i-2])\setminus \{i-1\}) \\
  &\vartriangleleft& (i-2,\sigma, (S\cup [j,i-3])\setminus \{i-2,i-1\}) \\
&\vartriangleleft ...\vartriangleleft& (j+1,\sigma, (S\cup \{j\})\setminus [j+1,i-1]) \\
&\vartriangleleft & (j,\sigma, S\setminus [j,i-1]) \\
&\leqslant & (j,\sigma, (S\setminus [j,i-1])\cup \{a_1,b_1\}) \\
&\vartriangleleft & (j,\sigma_1, (S\setminus [j,i-1])\setminus \{a_1,b_1\}) \\
&\leqslant & (j,\sigma_1, (S\setminus ([j,i-1]\cup\{a_1,b_1\}))\cup \{a_2,b_2\}) \\
&\vartriangleleft & (j,\sigma_2, (S\setminus [j,i-1])\setminus \{a_1,b_1,a_2,b_2\}) \\
& \leqslant ... \leqslant & (j,\tau,(S\setminus [j,i-1])\setminus \{a_1,b_1,...,a_n,b_n\}) \leqslant (j,\tau,T),
 \end{eqnarray*} since $\{a_1,b_1,...,a_n,b_n\} =[m]\setminus F(\sigma,\tau)$.
The length formula follows by Proposition \ref{Dposet2}.

The last statement follows immediately noting that the map $(\sigma, S) \mapsto (m+1,\sigma, S)$ is a poset isomorphism between $\W(S_{2m})$
and $\{ (i,\sigma, S) \in \W(S_{2m+1}) : i=m+1 \}$.
\end{proof}


We illustrate the preceding theorem with an example.
  Let $u=(4,2431,\{1,2,3\})\in \W(S_9)$ and $v=(3,3421,\{2\}) \in \W(S_9)$. Then we have that $2431<3421$ and $S\cap ([j-1] \cup [i,m]) \cap F(\sigma,\tau)=\{1,2,3\}\cap (\{1,2\} \cup \{4\}) \cap \{2,4\}=\{2\}$;
hence, by Theorem \ref{bruhat-odd}, $u<v$.

The following lemma can be easily deduced by Theorem \ref{bruhat-odd} so we omit its verification.
\begin{lem} \label{lemma unione}
Let $m>0$ and $u,v\in \W(S_{2m+1})$. If $u\leqslant (i,\sigma, S_1) \leqslant v$ and $u\leqslant (i,\sigma, S_2) \leqslant v$ in $(\W(S_{2m+1}),\leqslant)$ then
$u\leqslant (i,\sigma, S_1 \cup S_2) \leqslant v$.
\end{lem}

The characterization obtained in Theorem \ref{bruhat-odd} enables us to give an explicit expression for the M\"obius function of lower intervals in the poset of Wachs permutations partially ordered by Bruhat order, and shows,
in particular, that it has values in $\{ 0,1,-1 \}$.
\begin{pro} \label{moebius-odd}
  Let $m>0$, and $v=(j,\tau,T) \in \W(S_{2m+1})$.
Then
$$ \mu(e,v)=\left\{
              \begin{array}{ll}
                (-1)^{|T|}, & \hbox{if $\tau=e$ and $j=m+1$;} \\
                0, & \hbox{otherwise.}
              \end{array}
            \right.
$$
In particular, if $v=(\tau,T)\in \W(S_{2m})$ then
$$ \mu(e,v)=\left\{
              \begin{array}{ll}
                (-1)^{|T|}, & \hbox{if $\tau=e$;} \\
                0, & \hbox{otherwise.}
              \end{array}
            \right.
$$
\end{pro}
\begin{proof}
We proceed by induction on $\ell_{\W}(v)$.
  If $\tau=e$ and $j=m+1$ then, by Theorem \ref{bruhat-odd}, the interval $[e,v]$ is isomorphic to a Boolean algebra, so we conclude that $\mu(e,v)=(-1)^{|T|}$, as desired. So assume that either
 $\tau \neq e$ or $j <m+1$. Then, by Proposition \ref{cover-odd}, $\ell_{\W}(v) \geq 2$, so,
 by Lemma \ref{lemma unione} there exists $R\subseteq [m]$, $R \neq \emptyset$,
such that $[e,v] \cap \{(k,\rho, U)\in \W(S_{2m+1}):\rho =e,k=m+1\}=[e,(m+1,e,R)]$. Hence
\begin{eqnarray*}
  \mu(u,v) &=& -\sum \limits_{x\in [e,v)} \mu(e,x) \\
  &=& -\sum \limits_{x\in [e,(i,\sigma,R)]} \mu(e,x) -\sum \limits_{x\in [e,v)\setminus [e,(i,\sigma,R)]} \mu(e,x)\\
  &=& -\sum \limits_{x\in [e,v)\setminus [e,(i,\sigma,R)]} \mu(e,x)=0
\end{eqnarray*}
by our induction hypothesis, and the fact that
$|[e,(i,\sigma,R)]| \neq 1$, where $[u,v):=\{z \in \W(S_{2m+1}): u \leqslant z < v\}$.

The statement about Wachs permutations in the even case follows from the odd one as in the proof of Theorem \ref{bruhat-odd}.
\end{proof}



We conclude by computing, using Proposition \ref{moebius-odd}, the characteristic polynomial of the poset of Wachs permutations.

\begin{cor}
\label{charpolyA}
  The characteristic polynomial of $(\W(S_n),\leqslant)$ is
$$
(x-1)^{\left\lfloor\frac{n}{2}\right\rfloor}x^{\binom{n}{2}-\binom{\left\lfloor\frac{n}{2}\right\rfloor+1}{2}},
$$
for all $n\in \mathbb{P}$.
\end{cor}
\begin{proof}
  The result follows from Poposition \ref{moebius-odd} and
  Theorem \ref{Dposet-odd}.
\end{proof}

The following is a commutative diagram that summarizes the poset morphisms considered in this section. The function $\pi_i$ stands for the canonical projection on the $i$-th factor of a Cartesian product.
Notice that, if $A$ and $B$ are posets,  the projection $\pi_1: A \otimes B \rightarrow A$ is order preserving, whereas $\pi_2: A \otimes B \rightarrow B$ is not order preserving.
\begin{center}
  \begin{tikzcd} {[m+1]^*} \times S_m \times \mathcal{P}({[m]})\arrow[r, "\phi^{-1}_{2m+1}"]\arrow[d, "\pi_2 \times \pi_3"]& \W(S_{2m+1}) \arrow[r, "\phi_{2m+1}"]
  \arrow[dd, bend left=70, "f_{2m+1}" near start]&
\left({[m+1]^*} \times S_m \right)\otimes \mathcal{P}({[m]})\\S_m \times \mathcal{P}({[m]}) \arrow[r, "\phi^{-1}_{2m}" description]\arrow[rd, shift right=2, "\pi_1" description] & \W(S_{2m})
\arrow[r, crossing over, "\phi_{2m}" description]\arrow[d, "f_{2m}"] & S_m \otimes \mathcal{P}({[m]})\arrow[ld, shift left=2,  "\pi_1" description] \\ & S_m &\end{tikzcd}
\end{center}

\section{Signed Wachs permutations and Bruhat order}

For $n>0$ recall (see \cite{BS}) that the set of \emph{signed Wachs permutation} is
$$\mathcal{W}(B_n) := \{ \sigma \in B_n : | \sm(i)-\sm(i^{\ast})| \leqslant 1 \; \forall \;i \in [n-1]  \}.$$
So, for example, $[-2,-1,4,3] \in \mathcal{W}(B_4)$ while $[3,4,-2,1] \notin \mathcal{W}(B_4)$. 
In the even case, as in type $A$, we have the following group isomorphism (see Proposition \ref{prop isomorfismo GI})
$$\W(B_{2m}) \simeq B_m \ltimes {\cal P}([m])=S_2 \wr B_m.
$$

We define a bijection $\phi: \mathcal{W}(B_{2m}) \rightarrow B_m \times {\cal P}([m])$ as follows. For $\sigma \in B_m$ and $T\subseteq [m]$ let $\phi^{-1}(\sigma,T):=v$, where $v\in \mathcal{W}(B_{2m})$ is defined by
$$v(2i-1)=\left\{
    \begin{array}{ll}
      2\sigma(i)-\chi(\sigma(i)>0), & \hbox{if $i\not \in T$,} \\
     2\sigma(i)+\chi(\sigma(i)<0), & \hbox{if $i\in T$,}
    \end{array}
  \right.
$$ and
$$v(2i)=\left\{
    \begin{array}{ll}
      2\sigma(i)+\chi(\sigma(i)<0), & \hbox{if $i\not \in T$,} \\
      2\sigma(i)-\chi(\sigma(i)>0), & \hbox{if $i \in T$,}
    \end{array}
  \right.
$$ for all $i \in [m]$.
For example, let $v:=[-3,-4,1,2,6,5]\in \mathcal{W}(B_6)$; then $\phi(v)=([-2,1,3],\{1,3\})$.
Because of this bijection from now on we freely identify the sets $\mathcal{W}(B_{2m})$ and $B_m \times {\cal P}([m])$, so if $v \in \mathcal{W}(B_{2m})$ and $\phi(v)=(\sigma, T)$, then we simply write $v=(\sigma, T)$ and we define $$\ell_{\mathcal{W}}(v):= \ell_B(v)-\ell_B(\sigma).$$

Recall that we denote by $v \mapsto \tilde{v}$ the natural embedding $B_n \hookrightarrow S_{\pm n}$. Note that if $v \in \mathcal{W}(B_{2m})$ then $\tilde{v} \in \mathcal{W}(S_{\pm 2m})$.
Indeed, if $v =(\sigma, T)$ 
then 
$\tilde{v}=(\tilde{\sigma}, -T \cup T)$. In fact, by Proposition \ref{prop Bruhat B}, this is an injective group and poset morphism $\mathcal{W}(B_{2m}) \hookrightarrow \mathcal{W}(S_{\pm 2m})$.
For example, for $u=[-2,-1,6,5,-3,-4]=([-1,3,-2], \{ 2,3 \}) \in \mathcal{W}(B_6)$ we have
$$\tilde{u}=(4,3,-5,-6,1,2,-2,-1,6,5,-3,-4)=((2,-3,1,-1,3,-2),\{-3,-2,2,3\})\in \mathcal{W}(S_{\pm 6}).$$
Notice that if $n$ is odd then the image of a signed Wachs permutation is not a Wachs permutation.
For example, if $u=[-2,-1,6,5,-3,-4,7]\in \mathcal{W}(B_7)$ we have
$$\tilde{u}=(-7,4,3,-5,-6,1,2,-2,-1,6,5,-3,-4,7)\not \in \mathcal{W}(S_{\pm 7}).$$
It is known that $2 \ell_B(v)= \ell_A(\tilde{v})+\negg(v)$ (see e.g. \cite[Exercise 8.2]{BB}) so we have that
\begin{equation}\label{length W B-even}
\ell_\mathcal{W}(v)= \frac{\ell_{\mathcal{W}}(\tilde{v})+\negg(\sigma)}{2},\end{equation}
for all $v=(\sigma,T) \in \mathcal{W}(B_{2m})$, because $2 \negg(\sigma)=\negg(v)$.

\begin{pro}
   \label{morfismo B-even}
  The function $\W(B_{2m}) \rightarrow B_m$ defined by the assignment $(\tau, T) \mapsto \tau$ is order preserving.
\end{pro}
\begin{proof}
 Let $u=(\sigma,S)\in \W(B_{2m})$ and $v=(\tau,T) \in \W(B_{2m})$. We have that $u \leqslant v$ implies $\tilde{u} \leqslant \tilde{v}$ and then, by Corollary \ref{cor-morph-odd}, $\tilde{\sigma} \leqslant  \tilde{\tau}$. By Proposition \ref{prop Bruhat B}, this implies $\sigma \leqslant \tau$.
\end{proof}

It is easy to characterize, using Theorem \ref{bruhat-odd},
the Bruhat order relation between signed Wachs permutations in the
even case.

\begin{pro} \label{cor Bruhat B}
	Let $u,v \in \mathcal{W}(B_{2m})$, $u=(\sigma,S)$, $v=(\tau,T)$. Then $u \leqslant v$ if and only if
$\sigma \leqslant \tau$ in $B_{m}$ and $S \cap F(\sigma,\tau) \subseteq T$, where
$F(\sigma,\tau):=\{ i \in [m] : \sigma(i)=\tau(i) \}$.
\end{pro}
\begin{proof}
We have that $\tilde{u}=(\tilde{\sigma},-S \cup S)$ and similarly
 $\tilde{v}=(\tilde{\tau},-T \cup T)$. But, by Theorem \ref{bruhat-odd},
 $\tilde{u} \leqslant \tilde{v}$ if and only if
$\tilde{\sigma} \leqslant \tilde{\tau}$ and $(-S \cup S) \cap F(\tilde{\sigma},\tilde{\tau})
\subseteq -T \cup T$, which in turn happens if and only if
$\sigma \leqslant \tau$ and $S \cap F(\sigma,\tau) \subseteq T$.
\end{proof}

The next result is the analogue of Proposition \ref{Dposet}.

\begin{pro} \label{Bposets}
	Let $m>0$. Then
	\begin{enumerate}
		\item $\phi: \mathcal{W}(B_{2m}) \rightarrow B_m \otimes \mathcal{P}([m])$ is order preserving;
		\item $\phi^{-1}: B_m \times \mathcal{P}([m]) \rightarrow \mathcal{W}(B_{2m})$ is order preserving.
	\end{enumerate} Moreover $\ell_B(\tau,T)=4\ell_B(\tau)+|T|-\negg(\tau)$, for all $(\tau,T)\in \mathcal{W}(B_{2m})$.
\end{pro}
\begin{proof} Points 1. and 2. are direct consequences of Propositions \ref{prop Bruhat B} and \ref{Dposet}.
  The last equality follows by the formula of Proposition \ref{Dposet}; in fact
\begin{eqnarray*}
  \ell_B(\tau,T) &=& \frac{\ell_A(\tilde{\tau},-T\cup T)+\negg(\tau,T)}{2}=2\ell_A(\tilde{\tau})+|T|+\negg(\tau) \\
   &=& 4\ell_B(\tau)-2\negg(\tau)+|T|+\negg(\tau).
\end{eqnarray*}\end{proof}

For any $(i,j)\in [n]\times [\pm n]$, $i\neq j$, we find it convenient to define $$(i,j)_B:=\left\{
                 \begin{array}{ll}
                   (i,j)(-i,-j), & \hbox{if $i\neq |j|$;} \\
                   (i,-i), & \hbox{otherwise.}
                 \end{array}
               \right.$$
So the set of reflections of $B_n$
is (see \cite[Proposition 8.1.5]{BB}) $$T^{B_n}=\{(i,j)_B : 1 \leqslant i < |j| \leqslant n\} \cup \{(i,-i)_B : i\in [n]\}.$$
Let $m>0$. For any reflection $(i,j)_B \in T^{B_m}$ we define an involution $w^B_{i,j}: \mathcal{W}(B_{2m}) \rightarrow \mathcal{W}(B_{2m})$ by letting
$$w^B_{i,j}(\tau,T):=  \left(\tau (i,j)_B, T+\{i,|j|\}\right),$$
for all $(\tau, T)\in \mathcal{W}(B_{2m})$, where $X+Y$ stands for the symmetric difference between two sets $X$ and $Y$.
For example, if $v=([-2,1,4,3],\{1,4\})$ then $w^B_{3,-3}(v)=([-2,1,-4,3],\{1,3,4\})$
and $w^B_{1,-3}(v)=([-4,1,2,3],\{3,4\})$.

%

\begin{rmk} \label{rmk tilde}
  We observe that, under the embedding $\mathcal{W}(B_{2m}) \hookrightarrow \mathcal{W}(S_{\pm 2m})$, we have that $w^B_{i,j}(v)\mapsto w^A_{-i,-j}(w^A_{i,j}(\tilde{v}))$, if $i\neq -j$, and $w^B_{i,-i}(v) \mapsto w^A_{i,-i}(\tilde{v})$.
\end{rmk}


The next technical result enables us to ``lift'' some order theoretic
properties from $B_{2m}$ to $\mathcal{W}(B_{2m})$. Its proof
relies on the corresponding result in type $A$, namely Lemma \ref{lemma-tij}.
\begin{lem} \label{lemmaB}
  Let $m>0$, $i\in [m]$ and $j\in [\pm m]$. Assume $v=(\tau,T)\in \mathcal{W}(B_{2m})$, $i,|j|\not \in T$ and $\tau(i,j)_B \vartriangleleft \tau$.
Then $w^B_{i,j}(v)< v$
  and $\ell_{\mathcal{W}}(w^B_{i,j}(v),v)=1$. Moreover if $u=(\sigma, S)\in \mathcal{W}(B_{2m})$
  and $\sigma \leqslant \tau(i,j)_B$ then  $u \leqslant w^B_{i,j}(v) < v$.
\end{lem}
\begin{proof}
  Consider the element $\tilde{v}=(\tilde{\tau}, -T \cup T)\in \mathcal{W}(S_{\pm 2m})$. Then by our hypothesis and Theorem \ref{incitti} we have that $\tilde{\tau}(i,j)(-i,-j) \vartriangleleft \tilde{\tau}(i,j) \vartriangleleft \tilde{\tau}$
  if $i\neq |j|$, and $\tilde{\tau}(i,-i) \vartriangleleft \tilde{\tau}$ if $i=|j|$. So, by Lemma \ref{lemma-tij},
  $\tilde{u} \leqslant w^A_{-i,-j}(w^A_{i,j}(\tilde{v})) < w^A_{i,j}(\tilde{v}) < \tilde{v}$ (since $i,j\not \in T$ implies $-i,-j\not \in -T$) if $i\neq |j|$,
  and $\tilde{u} \leqslant w^A_{i,-i}(\tilde{v})<\tilde{v}$ if $i=|j|$. By Proposition \ref{prop Bruhat B} and Remark \ref{rmk tilde} we conclude that
  $u \leqslant w^B_{i,j}(v)<v$.
  Note also that, by Theorem \ref{incitti}, $i\neq |j|$ and $v(i,j)_B \vartriangleleft v$ in $B_n$ imply $\negg(v(i,j)_B)=\negg(v)$, and that
$v(i,-i)_B \vartriangleleft v$ in $B_n$ implies $\negg(v(i,-i)_B)=\negg(v)-1$.
  Therefore, if $i,|j| \not \in T$ and $\tau(i,j)_B\vartriangleleft \tau$, then by Lemma \ref{lemma-tij} and \eqref{length W B-even} we have that $\ell_{\mathcal{W}}( w^B_{i,j}(v),v)=1$.
\end{proof}

The following result characterizes the cover relations of the ordering
induced by Bruhat order on the signed Wachs permutations in the even case.

\begin{cor} \label{coverB-even}
  Let $m>0$, $u,v\in \mathcal{W}(B_{2m})$, $u=(\sigma,S)$ and $v=(\tau,T)$. Then $u\vartriangleleft v$ if and only if
  either one of the following conditions is satisfied:
\begin{enumerate}
  \item \label{uno} $\sigma=\tau$ and $S\vartriangleleft T$;
  \item \label{due} $\sigma \vartriangleleft \tau$, $T\cap \{a,|b|\}= \varnothing$ and $S=T\cup \{a,|b|\}$, where $(a,b)_B:=\tau^{-1}\sigma$.
\end{enumerate}
\end{cor}
\begin{proof}
Let $u \vartriangleleft v$ in $\W(B_{2m})$. Then $\sigma \leqslant  \tau$ by Corollary \ref{morfismo B-even}.
Moreover, if $\sigma < \omega< \tau$ for some $\omega \in B_m$ then, by Proposition \ref{Bposets}, $u\leqslant(\omega,S)\leqslant v$.
So $u \vartriangleleft v$ implies $\sigma = \tau$ or $\sigma \vartriangleleft \tau$.

If $\sigma = \tau$ then, by Corollary \ref{cor Bruhat B}, $S \subseteq T$; since $u\vartriangleleft v$, we have that $S \vartriangleleft T$.
Assume now $\sigma \vartriangleleft  \tau$ and let $(a,b)_B:=\tau^{-1}\sigma$.
If $a\in T$ or $|b|\in T$ then, by point 2 of the proof of Theorem \ref{Dposet-odd}, $\tilde{u} \leqslant \tilde{v}(a,a+1)(-a-1,-a)<\tilde{v}$ and
$\tilde{u} \leqslant \tilde{v}(|b|,|b|+1)(-|b|-1,-|b|)<\tilde{v}$, respectively. Hence $u\vartriangleleft v$ and $\sigma \vartriangleleft \tau$ imply $T\cap \{a,|b|\}=\varnothing$ and
$u=w^B_{a,b}(v)$, by Lemma \ref{lemmaB}. This implies $S=T\cup \{a,|b|\}$.

The converse can be proved by noting that if condition 1 or 2 are satisfied then, by Proposition \ref{cover-odd},
$[\tilde{u},\tilde{v}]=\{\tilde{u},w_1,w_2,\tilde{v}\}$, with $w_1,w_2 \in \W(S_{\pm 2m})\setminus \W(B_{2m})$,
or $[\tilde{u},\tilde{v}]=\{\tilde{u},\tilde{v}\}$, where the intervals are taken in $\W(S_{\pm 2m})$.
\end{proof} For example
$[-2,-1,3,4,6,5,-7,-8] \vartriangleleft [-2,-1,4,3,6,5,-7,-8] \vartriangleleft [-2,-1,5,6,3,4,-7,-8]$.
By Corollary \ref{coverB-even} we have that if $u \vartriangleleft v$ holds in $\mathcal{W}(B_{2m})$ then either
$u=v(2i-1,2i)_B$ for some $i \in [m]$, or $u=w^B_{i,j}(v)$ for some $(i,j)_B\in T^{B_m}$.

We can now prove the even part of the main result of this section.

\begin{thm}
\label{Bevengraded}
The poset $\mathcal{W}(B_{2m})$ is graded, with rank function $\ell_{\mathcal{W}}$, and its rank is $3m^2$.
\end{thm}
\begin{proof} Since $e, w_0^{B_{2m}} \in \mathcal{W}(B_{2m})$, these are the minimum and maximum of the poset $\mathcal{W}(B_{2m})$, respectively.
Let $u=(\sigma,S) \in \mathcal{W}(B_{2m})$ and $v=(\tau,T) \in \mathcal{W}(B_{2m})$ be such that $u \vartriangleleft v$. Therefore, by Corollary \ref{coverB-even},
either $u=v(2i-1,2i)_B$ for some $i \in [m]$ or $u=w^B_{i,j}(v)$, where $\sigma=\tau(i,j)_B \vartriangleleft \tau$.
In both cases $\ell_{\mathcal{W}}(u,v)=1$, and this proves the first statement. The rank of the poset $\W(B_{2m})$ is, by \eqref{length W B-even}, $\ell_{\W}(w_0(B_{2m}))=3m^2$,
being $\ell_{\mathcal{W}}\left(\widetilde{w_0(B_{2m})}\right)=m(6m-1)$.
\end{proof}

We now investigate the Bruhat order on $\W(B_n)$ for $n$ odd.
Let $v \in \W(B_{2m+1})$. Note that there is a bijection between
$\W(B_{2m+1})$ and $\W(B_{2m}) \times [ \pm (m+1)]$ given by
$v \mapsto (\overline{v},(\pos(v)+\sgn(\pos(v)))/2)$ where
$$
\overline{v}(i)=\left\{
    \begin{array}{ll}
      v(i), & \hbox{if $i < v^{-1}(2m+1)$,} \\
      v(i+1), & \hbox{if $i \geq v^{-1}(2m+1)$,}
    \end{array}
  \right.
$$
for $i \in [2m]$. Combining this with the bijection between
$\W(B_{2m})$ and $B_m \times {\cal P}([m])$ explained at the beginning of this section we obtain
a bijection $\phi$ between $\W(B_{2m+1})$ and $ [ \pm (m+1)]\times B_m \times {\cal P}([m])$. If $v \in \W(B_{2m+1})$ and $(i,\sigma,S) \in
[ \pm (m+1)] \times B_m \times {\cal P}([m])$ correspond under this
bijection then we write $v=(i,\sigma,S)$, and we define
\begin{equation} \label{def-rankB-odd} \ell _{W} (v) := \ell _B (v) - \ell _B (\sigma ),
\end{equation}
and $\check{v} := (\sigma ,S)$ (so $\check{v} \in \W(B_{2m})$).
So, for example, if $v=[ -1,-2,5,6, -7,3,4, ]$ then $ v=(-5,[-1,3,2], \{ -1 \})$ so $\ell _W (v) =(9+3+7)-2=17$, and
$\check{v}=[-1,-2,5,6,3,4]$. Note that, if $u,v \in W(B_{2m+1})$ are such that
$u^{-1}(2m+1)=v^{-1}(2m+1)$ then
\begin{equation}
\label{evenodd}
\ell _W (u,v)=\ell _W( \breve{u}, \breve{v}).
\end{equation}

The next result is the analogue of Proposition \ref{Dposet2}.

\begin{pro} \label{Bposets-odd}
	Let $m>0$. Then
	\begin{enumerate}
		\item $\phi: \mathcal{W}(B_{2m+1}) \rightarrow [\pm(m+1)]^* \otimes B_m \otimes \mathcal{P}([m])$ is order preserving;
		\item $\phi^{-1}: [\pm(m+1)]^*\times B_m \times \mathcal{P}([m]) \rightarrow \mathcal{W}(B_{2m+1})$ is order preserving.
	\end{enumerate} Moreover $\ell_B(v)=4\ell_B(\tau)+|T|-\negg(\tau)+2(m-i+1)-3\chi(i<0)=\ell_B(\breve{v})+2(m-i+1)-3\chi(i<0)$, for all $v=(i,\tau,T)\in \mathcal{W}(B_{2m+1})$.
\end{pro}
\begin{proof}
  The first two points follow easily from Proposition \ref{Bposets}
 and the fact that if $u \leq v$ then $\pos(u) \geq \pos(v)$. The
 length equality is easy to check using our definitions and the well known fact
 (see, e.g., \cite[Prop. 8.1.1]{BB}) that
 $\ell_B(v)=\inv(v)+\negg(v)+\nsp(v)$ for all $v$ in the hyperoctahedral
 group.
\end{proof}

For $v \in \W(B_{2m+1})$ we define an element $c(v) \in \W(B_{2m+1})$
by
\begin{equation} \label{defc(v)}
    c(v) := \left\{
          \begin{array}{ll}
            v(-1,1)_B, & \hbox{if $j=-1$;} \\
            v(j+1,j+2)_B, & \hbox{if $j \neq -1$ and
            $v(j+1)>v(j+2)$;} \\
            v(j,j+2)_B, & \hbox{if $j \neq -1$ and
            $v(j+1)<v(j+2)$;}
          \end{array}
        \right.
\end{equation}
where $j:= v^{-1}(2m+1)$. For example, $c([-9,4,3,-6,-5,2,1,-8,-7])=[9,4,3,-6,-5,2,1,-8,-7]$,
$c([4,3,-6,-5,9,2,1,-8,-7])=[4,3,-6,-5,9,1,2,-8,-7]$ and $c([3,4,-9,1,2,6,5,-7,-8])=[-9,4,3,1,2,6,5,-7,-8]$.
Note that $\ell _{\mathcal{W}}(c(v),v)=1$.
The next result implies that $c(v)$ is the only coatom $u$ of $v$ such that $v^{-1}(2m+1) \neq u^{-1}(2m+1)$, and is the main technical tool
in our proof of the fact that $\W(B_n)$, under Bruhat order, is graded.

\begin{thm}
\label{propB}
  Let $m\in \mathbb{P}$, and $u,v\in D(B_{2m+1})$, $u<v$, be such that $v^{-1}(2m+1) <u^{-1}(2m+1)$.
  Then $u \leq c(v)$.
\end{thm}
\begin{proof}
Let $j:=v^{-1}(2m+1)$. We distinguish various cases according as to whether $j<-1$, $j=-1$, or $j>0$. Let, for brevity,
$z:=c(v)$. Given a signed permutation $w \in B_n$ and $j \in [\pm n]$ we find it convenient to denote by $w_j$ the increasing rearrangement of $ \{ w(-n), w(-n+1), \ldots, w(j) \}$.

Assume first that $j>0$ and $j+1 \in D_{R}(v)$. Let $a := v(j+2)$ so
$v(j+1)=a+1$. Then $v_k =z_k$ for all $k \in [ \pm (2m+1)] \setminus
\{ j+1,-j-2 \}$. Since $u,v,z$ are all Wachs permutations there are
$ x_1, \ldots , x_{\tilde {m}} , y_1, \ldots , y_{\tilde {m}+1} \in  [ \pm (2m)]$, where $\tilde{m}:= \frac{2m-1+j}{2}$, such that
$$
v_{j+1} = (-2m-1,x_1,x_1 +1,\ldots , x_p , x_p +1, a+1, x_{p+1},
x_{p+1}+1, \ldots , x_{\tilde{m}},  x_{\tilde{m}}+1, 2m+1)
$$
$$
z_{j+1} = (-2m-1,x_1,x_1 +1,\ldots , x_p,x_p +1,a,x_{p+1},x_{p+1}+1, \ldots , x_{\tilde{m}}, x_{\tilde{m}}+1, 2m+1)
$$
and
$$
u_{j+1}=(-2m-1,y_1 , y_1 +1, \ldots  , y_{\tilde{m}+1}, y_{\tilde{m}+1}+1),
$$
for some $0 \leq p \leq \tilde{m}$.
Since $u_{j+1} \leq v_{j+1}$ we have that $y_{p+1} \leq a+1$. But
$y_{p+1}$ and $a$ are of the same parity if they have the same sign so
$y_{p+1} \leq a$ and hence $u_{j+1} \leq z_{j+1}$.

Similarly, there are
$ x_1, \ldots , x_{\tilde {m} } , y_1, \ldots , y_{\tilde {m} }  \in [ \pm (2m)]$, where $\tilde{m} := \frac{2m-j-1}{2}$ such that
$$
v_{-j-2} = (x_1,x_1 +1,\ldots , x_p , x_p +1, -a, x_{p+1},
x_{p+1}+1, \ldots , x_{\tilde{m}},  x_{\tilde{m}}+1)
$$
$$z_{-j-2} = (x_1,x_1 +1,\ldots , x_p,x_p +1,-a-1,x_{p+1},x_{p+1}+1, \ldots , x_{\tilde{m}}, x_{\tilde{m}}+1)
$$
and
$$
u_{-j-2}=(-2m-1,y_1 , y_1 +1, \ldots  , y_{\tilde{m}}, y_{\tilde{m}}+1),
$$
for some $0 \leq p \leq \tilde{m}$.
Since $u_{-j-2} \leq v_{-j-2}$ we have that $y_{p} \leq x_p +1$. But
$x_{p} +1 < -a-1$ so
$y_{p} +1 \leq -a-1$ and hence $u_{-j-2} \leq z_{-j-2}$.

Suppose now that ($j>0$ and) $j+1 \notin D_R(v)$.
Consider first
$v_j , z_j , u_j ,v_{j+1}, z_{j+1}, u_{j+1}$. Since $v,u$, and $z$ are all Wachs permutations, and by our definition of $z$, there are $\{ x_1, \ldots , x_{\tilde{m}} \}_{<},$
$\{ y_1, \ldots , y_{\tilde{m}} \}_{<} \subseteq [\pm (2m)]$, where $\tilde{m}:= \frac{2m-1+j}{2}$ such that
$$ v_j = (-2m-1,x_1,x_1 +1,\ldots ,x_{\tilde{m}},
x_{\tilde{m}}+1,2m+1) $$
\[
z_j = (-2m-1,x_1,x_1 +1,\ldots , x_p,x_p +1,a+1,x_{p+1},x_{p+1}+1,
\ldots , x_{\tilde{m}},
x_{\tilde{m}}+1)
\]
\[
u_j = (-2m-1,y_1,y_1 +1,\ldots , y_q,y_q +1,b^{\ast},y_{q+1},y_{q+1}+1,\ldots ,y_{\tilde{m}},
y_{\tilde{m}}+1)
\]
\[
v_{j+1} = (-2m-1,x_1,x_1 +1,\ldots , x_p,x_p +1,a,x_{p+1},x_{p+1}+1, \ldots , x_{\tilde{m}},
x_{\tilde{m}+1},2m+1)
\]
\[
z_{j+1} = (-2m-1,x_1,x_1 +1,\ldots , x_p,x_p +1,a,a+1,x_{p+1},x_{p+1}+1, \ldots , x_{\tilde{m}},
x_{\tilde{m}}+1)
\]
\[
u_{j+1} = (-2m-1,y_1,y_1 +1,\ldots , y_q,y_q +1,b,b+1,y_{q+1},y_{q+1}+1,\ldots ,y_{\tilde{m}},
y_{\tilde{m}}+1)
\]
where $a:=v(j+1)$
($=z(j+1)$), $b^\ast := u(j)$, and $p,q \in [\tilde{m}]$.
By our hypothesis $u \leq v$ so $u_j \leq v_j$ and $u_{j+1} \leq v_{j+1}$ (componentwise). This easily implies that $u_j \leq z_j$ and
$u_{j+1} \leq z_{j+1}$ (componentwise) keeping in mind the fact that $x_r$ and $y_{s}$ have the same parity if they have the same sign for all $r,s \in [\tilde{m}]$ (so $y_s \leq x_r+1$
implies $y_s \leq x_r$ for all $r,s \in [\tilde{m}]$). For example, if $q<p$, then $y_p \leq a$ (since $u _{j+1} \leq v_{j+1})$, while $y_k \leq x_k+1$ (since $u_j \leq v_j$) so $y_k \leq x_k$ if $p+1 \leq k \leq \tilde{m}$. Similarly, if $q>p$, then $y_{p+1} \leq a $ (since $u_{j+1} \leq v_{j+1}$) while $b \leq x_{q+1}$ so
$b \leq x_q$, while $y_k \leq x_k+1$ (since $u_j \leq v_j$) so $y_k \leq x_k$ if $q+1 \leq k \leq \tilde{m} $, and $y_k \leq x_{k-1}+1$ (since $u_{j+1} \leq v_{j+1}$) so $y_k \leq x_{k-1}$ if $p+2 \leq k \leq q$. The case $p=q$ is even simpler, and is therefore omitted.

Consider now $u_{-j-2}, z_{-j-2}, v_{-j-2}$ and $u_{-j-1}, z_{-j-1}, v_{-j-1}$.Then reasoning as above we have that there are
$\{ x_1, \ldots , x_{\tilde{m}} \}_{<},
\{ y_1, \ldots , y_{\tilde{m}} \}_{<} \subseteq [\pm (2m)]$ , where $\tilde{m} := m- \frac{j+1}{2}$, such that
$$v_{-j-2} = (x_1,x_1 +1,\ldots ,x_p, x_p+1,-a-1,x_{p+1},x_{p+1}+1,\ldots ,  x_{\tilde{m}},
x_{\tilde{m}}+1)
$$
$$z_{-j-2} = (-2m-1,x_1,x_1 +1,\ldots ,  x_{\tilde{m}},
x_{\tilde{m}}+1)
$$
$$u_{-j-2} = (-2m-1,y_1,y_1 +1, \ldots ,y_{\tilde{m}},
y_{\tilde{m}}+1)
$$
$$v_{-j-1} = (x_1,x_1 +1,\ldots , x_p,x_p +1,-a-1,-a,x_{p+1},x_{p+1}+1, \ldots , x_{\tilde{m}},
x_{\tilde{m}}+1)
$$
$$z_{-j-1} = (-2m-1,x_1,x_1 +1,\ldots , x_p,x_p +1,-a,x_{p+1},x_{p+1}+1, \ldots , x_{\tilde{m}},
x_{\tilde{m}}+1)
$$
$$u_{-j-1} = (-2m-1,y_1,y_1 +1,\ldots , y_q,y_q +1,-b, y_{q+1},y_{q+1}+1,\ldots ,y_{\tilde{m}},
y_{\tilde{m}}+1)
$$
where
$p,q \in [\tilde{m}]$. As above, the fact that
 $u_{-j-2} \leq v_{-j-2}$ and $u_{-j-1} \leq v_{-j-1}$  easily implies that
$u_{-j-2} \leq z_{-j-2}$ and $u_{-j-1} \leq z_{-j-1}$. For example, if $q<p$, then  $y_k \leq x_k +1$ (since $u_{-j-2} \leq v_{-j-2}$) so
$y_k \leq x_k$ for $ 1 \leq k \leq p$
and hence $-b <y_{q+1} \leq x_{q+1}$.
 Similarly, if $q \geq p$.

\medskip

 Consider now the case $j<-1$ (so $j \leq -3$). Assume first that $j+1 \notin D_R(v)$. Consider $u_j , z_j ,v_j, u_{j+1},
 z_{j+1}, v_{j+1}$. If $u^{-1}(-2m-1)>j$ then we conclude exactly as in the case $j>0$. So assume that $u^{-1}(-2m-1)<j$. Then reasoning as above we conclude that there are $\{ x_1, \ldots , x_{\tilde {m} } \}_{<}$, $\{ y_1, \ldots , y_{\tilde {m} } \}_{<} \subseteq [ \pm (2m)]$,
 where $\tilde {m} := m + \frac{j+1}{2}$, such that
$$
v_{j} = (x_1,x_1 +1,\ldots ,  x_{\tilde{m}},  x_{\tilde{m}}+1, 2m+1)
$$
$$z_{j} = (x_1,x_1 +1,\ldots ,x_p,x_{p}+1,a+1, x_{p+1}, x_{p+1}+1,\ldots ,  x_{\tilde{m}}, x_{\tilde{m}}+1)
$$
$$u_{j} = (-2m-1,y_1,y_1 +1, \ldots ,y_{\tilde{m}},
y_{\tilde{m}}+1)
$$
$$v_{j+1} = (x_1,x_1 +1,\ldots , x_p,x_p +1,a,x_{p+1},x_{p+1}+1, \ldots , x_{\tilde{m}},
x_{\tilde{m}}+1, 2m+1)
$$
$$z_{j+1} = (x_1,x_1 +1,\ldots , x_p,x_p +1,a,a+1,x_{p+1},x_{p+1}+1, \ldots , x_{\tilde{m}},
x_{\tilde{m}}+1)
$$
$$u_{j+1} = (-2m-1,y_1,y_1 +1,\ldots , y_q,y_q +1,c, y_{q+1},y_{q+1}+1,\ldots ,y_{\tilde{m}},
y_{\tilde{m}}+1)
$$
$$v_{j+2} = (x_1,x_1 +1,\ldots , x_p,x_p +1,a,a+1,x_{p+1},x_{p+1}+1, \ldots , x_{\tilde{m}},
x_{\tilde{m}}+1, 2m+1)
$$
$$u_{j+2} = (-2m-1,y_1,y_1 +1,\ldots , y_q,y_q +1,c,c+1, y_{q+1},y_{q+1}+1,\ldots ,y_{\tilde{m}},
y_{\tilde{m}}+1)
$$
where
$a:=v(j+1)$
($=z(j+1)$), $c := u(j+1)$, and $p,q \in [\tilde{m}]$.
By our hypothesis $u \leq v$ so $u_j \leq v_j$, $u_{j+1} \leq v_{j+1}$, and $u_{j+2} \leq v_{j+2}$ (componentwise). As above, this implies that $u_j \leq z_j$ and $u_{j+1} \leq z_{j+1}$.
For example, if $p<q$, then $y_{p+1} \leq a+1$ (so $y_{p+1} +1 \leq
x_{p+1}$), $y_{p} +1 \leq a$, $c \leq x_{q}+1$, and $y_i \leq x_{i-1}+1$
(so $y_{i} \leq x_{i-1}$) for $p+2 \leq i \leq q$
(since $u_{j+2} \leq v_{j+2}$), while $y_{i} \leq x_i +1$ (so $y_i \leq x_i$) for $p+1 \leq i \leq \tilde{m}$ ( since $u_j \leq v_j$). Similarly, if $q<p$, then $y_{p} \leq a$ and $y_{i} \leq x_i +1$ (so
$y_{i} \leq x_i $) for $p+1 \leq i \leq \tilde{m}$
(since $u_{j+1} \leq v_{j+1}$). Finally, if $q=p$, then $c \leq a+1$,
$y_{p+1} \leq a$ and $y_i \leq x_i +1$ (so $y_i \leq x_i$) for $p+1 \leq i \leq \tilde{m} $ (since $u_{j+2} \leq v_{j+2}$).

Consider now  $u_{-j-1} , v_{-j-1}, z_{-j-1}, u_{-j-2}, v_{-j-2}, z_{-j-2}$. If $u^{-1}(-2m-1)>j$ then we conclude exactly as in the case $j>0$. So assume that $u^{-1}(-2m-1)<j$. Then  as above, since $u,v$, and $z$ are all Wachs permutations we conclude that  there are $\{ x_1, \ldots , x_{\tilde {m} } \}_{<} \subseteq [\pm (2m)]$ and $\{ y_1, \ldots , y_{\tilde {m} } \}_{<} \subseteq [ \pm (2m)]$ such that
$$
v_{-j-2} = (x_1,x_1 +1,\ldots , x_p , x_p +1, -a-1, x_{p+1},
x_{p+1}+1, \ldots , x_{\tilde{m}},  x_{\tilde{m}}+1, 2m+1)
$$
$$z_{-j-2} = (-2m-1, x_1,x_1 +1,\ldots  ,  x_{\tilde{m}}, x_{\tilde{m}}+1, 2m+1)
$$
$$u_{-j-2} = (-2m-1,y_1,y_1 +1, \ldots ,y_q , y_{q}+1, c^{\ast}, y_{q+1},
y_{q+1}+1, \ldots , y_{\tilde{m}}, y_{\tilde{m}}+1)
$$
$$v_{-j-1} = (x_1,x_1 +1,\ldots , x_p,x_p +1,-a-1,-a,x_{p+1},x_{p+1}+1, \ldots , x_{\tilde{m}}, x_{\tilde{m}}+1, 2m+1)
$$
$$z_{-j-1} = (-2m-1, x_1, x_1+1, \ldots , x_p,x_p +1,-a,x_{p+1}, x_{p+1}+1, \ldots , x_{\tilde{m}}, x_{\tilde{m}}+1, 2m+1)
$$
$$u_{-j-1} = (-2m-1,y_1,y_1 +1,\ldots , y_q,y_q +1,c,c+1, y_{q+1},y_{q+1}+1,\ldots ,y_{\tilde{m}},
y_{\tilde{m}}+1)
$$
$$v_{-j-3}=(x_1 , x_1 +1, \ldots  , x_{\tilde{m}}, x_{\tilde{m}}+1,2m+1)
$$
$$u_{-j-3}=(-2m-1,y_1 , y_1 +1, \ldots  , y_{\tilde{m}}, y_{\tilde{m}}+1)
$$
where $\tilde{m} := m - \frac{j+3}{2}$, $-a := v(-j-1)$ ($=z(-j-1)$), and $c^{\ast } := u(-j-2)$, for some $p,q \in [ \tilde{m}]$. It is then  not hard to conclude that $u_{-j-2} \leq z_{-j-2}$ and $u_{-j-1} \leq z_{-j-1}$.
For example, if $p \leq q$, then $y_i \leq x_i +1$ (since $u_{-j-2} \leq v_{-j-2}$) so $y_{i} \leq x_{i}$ for $1 \leq i \leq p$.
If $p>q$ then $y_i \leq x_i +1$ ( since $u_{-j-3} \leq v_{-j-3}$) so
$y_i \leq x_i$ for $1 \leq i \leq \tilde{m}$ and hence $c+1 < y_{q+1} \leq x_{q+1}$.

Assume now that ($j<-1$ and) $j+1 \in D_R(v)$.
Then $v_k = z_k$ if
$k \in [ \pm (2m+1) ] \setminus \{ -j-2,j+1 \}$. Assume first that
$u^{-1} (2m+1) \geq -j$. Let $a:= v(j+2)$ so $v(j+1)=a+1$. Then there are $x_1, \ldots , x_{\tilde {m} } , y_1, \ldots , y_{\tilde {m} }  \in [ \pm (2m)]$, where $\tilde{m} := \frac{2m-3-j}{2}$ such that
\begin{equation}
\label{v-j-2}
v_{-j-2} = (x_1,x_1 +1,\ldots , x_p , x_p +1, -a, x_{p+1},
x_{p+1}+1, \ldots , x_{\tilde{m}},  x_{\tilde{m}}+1, 2m+1)
\end{equation}
\begin{equation}
\label{z-j-2}
z_{-j-2} = (x_1,x_1 +1,\ldots , x_p,x_p +1,-a-1,x_{p+1},x_{p+1}+1, \ldots , x_{\tilde{m}}, x_{\tilde{m}}+1,2m+1)
\end{equation}
and
$$u_{-j-2}=(-2m-1,y_1 , y_1 +1, \ldots , y_q,y_q +1, c^{\ast }, y_{q+1},
y_{q+1}+1, \ldots ,  y_{\tilde{m}}, y_{\tilde{m}}+1),
$$
for some $0 \leq p,q \leq \tilde{m}$ where
$c^{\ast } := u(-j-2)$. It is then easy to see that $u_{-j-2} \leq z_{-j-2}$. Indeed, if $q<p$ then, since $u_{-j-2} \leq v_{-j-2}$,
$y_p \leq -a$ so, as above, $y_p \leq -a-1$ and hence $u_{-j-2} \leq z_{-j-2}$. If $q \geq p$ then $y_p \leq x_p +1$ (since
$u_{-j-2} \leq v_{-j-2}$) so $y_p \leq x_p$. But $x_p +1<-a$ hence $y_p +1 \leq -a-1$ so  $u_{-j-2} \leq z_{-j-2}$.

Similarly, there are
$\{ x_1, \ldots , x_{\tilde {m} }\}_< , \{y_1, \ldots , y_{\tilde {m} } \}_{<} \in [ \pm (2m)]$, where $\tilde{m} := \frac{2m+j+1}{2}$, such that
\begin{equation}
\label{v_j+1}
v_{j+1} = (x_1,x_1 +1,\ldots , x_p , x_p +1, a+1, x_{p+1},
x_{p+1}+1, \ldots , x_{\tilde{m}},  x_{\tilde{m}}+1, 2m+1)
\end{equation}
\begin{equation}
\label{zj+1}
z_{j+1} = (x_1,x_1 +1,\ldots , x_p,x_p +1,a,x_{p+1},x_{p+1}+1, \ldots , x_{\tilde{m}}, x_{\tilde{m}}+1,2m+1)
\end{equation}
and
$$u_{j+1}=(-2m-1,y_1 , y_1 +1, \ldots , y_q,y_q +1, c^{\ast }, y_{q+1},
y_{q+1}+1, \ldots ,  y_{\tilde{m}}, y_{\tilde{m}}+1),
$$
for some $0 \leq p$, $q \leq \tilde{m}$, where
$c^{\ast } := u(j+1)$, and we conclude exactly as in the last case.

Assume now that
$u^{-1} (2m+1) < -j$. Then there are $ x_1, \ldots , x_{\tilde {m} } , y_1, \ldots , y_{\tilde {m}+1 }  \in [ \pm 2m]$, where $\tilde{m} := \frac{2m+1+j}{2}$ such that
(\ref{v_j+1}) and (\ref{zj+1}) hold
and
$$u_{j+1}=(y_{1},y_{1}+1, \ldots ,  y_{\tilde{m}+1}, y_{\tilde{m}+1}+1)
$$
and the result follows easily. Similarly, there are
$ x_1, \ldots , x_{\tilde {m} } , y_1, \ldots , y_{\tilde {m} }  \in [ \pm 2m]$, where $\tilde{m} := \frac{2m-3-j}{2}$, such that
(\ref{v-j-2}) and (\ref{z-j-2}) hold
and
$$u_{-j-2}=(-2m-1,y_1 , y_1 +1, \ldots ,  y_{\tilde{m}}, y_{\tilde{m}}+1, 2m+1)
$$
for some $0 \leq p \leq \tilde{m}$. Hence, since
$u_{-j-2} \leq v_{-j-2}$, $y_{p} \leq x_p +1<-a-1$ so
$y_{p} +1 \leq -a-1$.

Finally, if $j=-1$ then $z= v(-1,1)$ so there are $x_1, \ldots , x_m$, $y_1, \ldots , y_m \in [ \pm 2m]$ such that $v_{-1}=(x_1 , x_1 +1, \ldots  , x_{m}, x_{m}+1,2m+1)$,
$z_{-1}=(-2m-1,x_1 , x_1 +1, \ldots  , x_{m}, x_{m}+1)$ and
$u_{-1}=(-2m-1,y_1 , y_1 +1, \ldots  , y_{m}, y_{m}+1)$,
so the conclusion follows easily. This concludes the proof.
\end{proof}

Recall that if $v=(\sigma,S,i) \in \W(B_{2m+1})$ then we let $\breve{v}=(\sigma,S)\in \W(B_{2m})$. We can now prove the main result of this section.
\begin{thm} \label{mainB}
$\W(B_{n})$ is graded, with rank function $\ell_{\W}$, and its rank is
$n^2-\left\lfloor\frac{n}{2}\right\rfloor^2$.
\end{thm}
\begin{proof}
If $n$ is even then this follows from Theorem
\ref{Bevengraded}. So assume that $n=2m+1$ for some $m>0$. Since $\W(B_{2m+1})$ has both a maximum and a minimum element it is enough to show that if $u,v \in \W(B_{2m+1})$ and $u \lhd v$ then
$\ell _{\W} (u,v)=1$. So let $u,v \in \W(B_{2m+1})$ be such that $u \lhd v$.
Let  $j:= u^{-1}(2m+1)$, and $i:= v^{-1}(2m+1)$.
Since $u<v$ we have that $i \leq j$. If $i<j$  then by Theorem \ref{propB} we have that $u \leq c(v)<v$ so $u=c(v)$ and hence $\ell _{\W}( u,v)=1$. If $i=j$ then $\breve{u} \lhd \breve{v}$ so, by Theorem \ref{Bevengraded}, $\ell _{\W} (\breve{u}, \breve{v})=1$ and
hence, by (\ref{evenodd}), $\ell _{\W} (u,v)=1$.
Finally, it is not difficult to see, by our definition of $\ell_{\W}$, that $\ell_{\W}(w_0(B_{2m+1}))=3m^2+4m+1$.
\end{proof}

We remark that the sequence $\{\ell_{\W}(w_0(B_{2m}))\}_{m \in \mathbb{P}}$ gives the number of edges of the complete tripartite graph $K_{m,m,m}$ (see A033428 in OEIS), and
$\{\ell_{\W}(w_0(B_{2m+1}))\}_{m \in \mathbb{P}}$ is the sequence of octagonal numbers (see A000567 in OEIS). We can now compute the rank-generating function of $(\W(B_n),\leqslant)$.
\begin{cor} \label{poincareBn} Let $m>0$. Then
$$\W(B_{2m})(x,\ell_{\W})=(1+x)^m [m]_{x^3}! \prod\limits_{i=1}^m(1+x^{3i-1})$$ and
  $$\W(B_{2m+1})(x,\ell_{\W})=[m+1]_{x^2}(1+x^{2m+1})(1+x)^m [m]_{x^3}! \prod\limits_{i=1}^m(1+x^{3i-1}).$$
\end{cor}
\begin{proof}
Note first that by our definition and Proposition \ref{Bposets} if
$u \in \W(B_{2m})$, $u=(\sigma,S)$, then $\ell_{\W}(u)=
3 \ell_B(\sigma)+|S|- \negg(\sigma)$. Therefore
$$
\sum_{u \in \W(B_{2m})} x^{\ell_{\W}(u)}=
\sum_{\sigma \in B_m} \sum_{S \subseteq [m]} x^{3 \ell_B(\sigma)+|S|- \negg(\sigma)}
= (1+x)^m \sum_{\sigma \in B_m} x^{3 \ell_B(\sigma)- \negg(\sigma)}.
$$
Let $J:=[m]$. Then by \cite[Prop. 2.4.4]{BB}
every element of $\sigma \in B_m$ may be expressed in a unique way as
$\sigma= zw$ where $w \in (B_m)_J$ and $z \in (B_m)^J$.
But the elements of $(B_m)_J$ are permutations of $S_m$
and the $z \in (B_m)^J$ are characterized by the fact that $z(1) < z(2) < \cdots < z(m)$. Hence these elements $z$ are in
bijection with subsets $S \subseteq [m]$, where the subset $S$ is the set of negative values taken by $\sigma$.
Furthermore, we then have that $\ell_B(\sigma)= \inv(\sigma) +\negg(\sigma) + \nsp(\sigma)=
\inv(w) +\negg(z) + \nsp(z)=\inv(w) + |S| + \sum_{s \in S} (s-1)
=\inv(w) + \sum_{s \in S} s$.
We therefore have that
\begin{eqnarray*}
    \sum_{\sigma \in B_m} x^{3 \ell_B(\sigma)- \negg(\sigma)}=
    \sum_{w \in S_m} \sum_{S \subseteq [m]} x^{3 \inv(w) -|S| +3 \sum_{s \in S} s }=
    \sum_{w \in S_m} x^{3 \inv(w)}
    \sum_{S \subseteq [m]} x^{\sum_{s \in S} (3s-1) }
\end{eqnarray*}
and the first equation follows. For the second formula we have, using Proposition \ref{Bposets-odd} and \eqref{def-rankB-odd},
\begin{eqnarray*}
\sum_{u \in \W(B_{2m+1})} x^{\ell_{\W}(u)}&=& \left(\sum\limits_{i=-m-1}^{-1}x^{2m-2i-1}+ \sum\limits_{i=1}^{m+1}x^{2(m-i+1)}\right)\W(B_{2m})(x, \ell_{\W}) \\
&=& [m+1]_{x^2}(1+x^{2m+1})\W(B_{2m})(x, \ell_{\W}).
\end{eqnarray*}
\end{proof}

Theorem \ref{mainB} enables us to explicitly determine the cover
relations in $\W(B_n)$ for $n$ odd.

\begin{cor} \label{coverB-odd}
	Let $u,v \in \W(B_{2m+1})$, $u=(i,\sigma,S)$, $v=(j,\tau,T)$ ($\sigma,\tau \in B_m$, $S,T \subseteq [m]$,
	$i,j \in [\pm (m+1)]$). Then $u \lhd v$ if and only if either one of the
	following conditions is satisfied:
	\begin{enumerate}
		\item $i=j$, $\sigma=\tau$, and $S \lhd T$;
		\item $j=i-1$, $\sigma=\tau$ and either $|j+\chi(j<0)| \in S$,
		$T= S \setminus \{ |j+\chi(j<0)| \}$ if $j \neq -1$, or $S=T$ if $j=-1$;
		\item $i=j$, $\sigma \lhd \tau$, $T \cap \{ |a|,|b| \}= \varnothing$
		and $S= T \cup  \{ |a|,|b| \}$, where $(a,b)_B=\tau \sigma^{-1}$.
	\end{enumerate}
\end{cor}
\begin{proof}
	Note that, since $\W(B_{2m+1})$ is graded, and its rank function is $\ell_{\W}$, $u \lhd v$ if and only if $u \leq v$ and
	$\ell_{\W}(u,v)=1$. If either 1. or 2. hold then it is easy to check that $u \leq v$ and $\ell _W (u,v)=1$.
	
	Assume now that 3. holds. Suppose first that $\tau =\sigma (a,b)(-a,-b)$. Since $\sigma \lhd \tau$ we have from Theorem \ref{incitti} that $ab>0$. We may assume that $0<a<b$. Again by Theorem \ref{incitti} we have that $\{ k \in [a+1,b-1]: \; \sigma (a) < \sigma (k) < \sigma (b) \} = \varnothing$. This, by 3., implies that $\inv (v)- \inv (u) =4-2$, $\negg (v)=\negg (u)$, and $\nsp (v)=\nsp (u)$. Hence $\ell_{\W} (u,v)=\ell _B (u,v)-\ell _B (\sigma , \tau )=2-1=1$. Suppose now that $\tau = \sigma (c,-c)$. We may clearly assume that $c>0$. Then by Theorem \ref{incitti}
	$\{ k \in [c-1]: \; - \sigma (c) < \sigma (k) < \sigma (c) \} = \varnothing$. This, again by 3., implies that $\inv (v) - \inv (u)=-2 ( \sigma (c)-1)-1$, $\negg (v)=\negg (u)+2$, and $\nsp (v)= \nsp (u)+2 (\sigma (c)-1)+1$. Hence $\ell _{\W} (u,v)=\ell _B (u,v)- \ell _B (\sigma , \tau )=2-1=1$.
	
	Conversely, assume that  $u \lhd v$. If $j>i$ then by Theorem \ref{propB} $u \leq c(v) <v$ so $u=c(v)$. If $j=-1$ then 2. follows easily from (\ref{defc(v)}). If $j \neq -1$ then $v(\pos(v) +1)< v(\pos(v)+2)$ (else, by \eqref{defc(v)}, $v^{-1}(2m+1)=c (v)^{-1} (2m+1)$ so $i=j$, a contradiction) and 2. again follows from \eqref{defc(v)}. If $i=j$ and $\sigma = \tau $ then it follows
	easily from Proposition \ref{prop Bruhat B} and Theorem \ref{tableau} that $S \subseteq T$ and 1. follows since $u \lhd v$.
	So assume ($i=j$ and) $\sigma < \tau$.
	Then, since $i=j$, $(\sigma , S) \lhd ( \tau , T)$ in $\W(B _{2m})$ so 3. follows from Corollary 4.6.

\end{proof}
So, for example, in $\W(B_9)$ we have that
$[9,2,1,-3,-4,8,7,-6,-5] \lhd [-9,2,1,-3,-4,8,7,-6,-5]$ $ \lhd [1,2,-9,-3,-4,8,7,-6,-5] \lhd [1,2,-9,-3,-4,8,7,-5,-6] \lhd [1,2,-9,-6,-5,8,7,-4,-3]$, where
the first two covering relations are of type 2. (with $j=-1$, and $j=-2$, respectively), the third one is of type 1., and the fourth one of type 3., with $(a,b)_B=(2,4)(-2,-4)$.

The next theorem characterizes the Bruhat order on signed Wachs permutations, in the odd case. Note that it generalizes and puts in perspective the results of Proposition \ref{Bposets-odd}.
\begin{thm} \label{Bbruhat-odd}
  Let $m>0$ and $u,v\in \W(B_{2m+1})$, $u=(i,\sigma,S)$, $v=(j,\tau, T)$. Then
$u\leqslant v$ if and only if

$$\mbox{$\sigma \leqslant \tau$, $\,$ $S(u,v) \subseteq T(u,v)$, and $\,$ $j\leqslant i$ },$$
where, for $X\subseteq [m]$, $X(u,v):=X \cap \left([\min\{|i|,|j|\}-1] \cup [\max\{|i|,|j|\},m] \right) \cap F(\sigma,\tau)$, being $F(\sigma,\tau):=\{i\in [m]:\sigma(i)=\tau(i)\}$. Moreover $\ell_{\W}(u)=3\ell_B(\sigma)+|S|-\negg(\sigma)+2(m-i+1)-3\chi(i<0)$.
\end{thm}
\begin{proof}
  Let $u\leqslant v$.
  We may assume $u\vartriangleleft v$. It is clear from Corollary \ref{coverB-odd} that $\sigma \leqslant \tau$ and $j \leqslant i$. Furthermore in case 1 of Corollary \ref{coverB-odd} and in case 2 with $j=-1$ we have that $S \subseteq T$ and then $S(u,v) \subseteq T(u,v)$.
  In case 2 with $j\neq -1$ we have that $i$ and $j$ have the same sign and $F(\sigma,\tau)=[m]$;
  hence $T=S\setminus \{j\}$ if $i>0$, so $T(u,v)=\left(S\setminus \{j\}\right) \cap \left([j-1] \cup [j+1,m]\right)=S(u,v)$. If $i<0$
  then $T=S\setminus \{|i|\}$ and $T(u,v)=\left(S\setminus \{|i|\}\right) \cap \left([|i|-1] \cup [|i|+1,m]\right)=S(u,v)$.
  Finally in case 3 we have that $F(\sigma,\tau)=[m]\setminus \{|a|,|b|\}$;
  since $S=T \cup \{|a|,|b|\}$ we obtain $T(u,v)=S(u,v)$.

 Now let $\sigma \leqslant \tau$, $j\leqslant i$ and $S(u,v) \subseteq T(u,v)$. Assume $j\leqslant i<0$.
 We then have the following chain in $(\W(B_{2m+1}),\leqslant)$:
\begin{eqnarray*}
  (i,\sigma, S) &\leqslant& (i,\sigma, S\cup [|i|,|j|-1]) \\
  &\vartriangleleft& (i-1,\sigma, (S\cup [|i|+1,|j|-1])\setminus \{|i|\}) \\
  &\vartriangleleft& (i-2,\sigma, (S\cup [|i|+2,|j|-1])\setminus \{|i|,|i|+1\}) \\
&\vartriangleleft ...\vartriangleleft& (j,\sigma, S \setminus [|i|,|j|-1])
\leqslant (j,\tau,T),
 \end{eqnarray*} where the last inequality follows from Proposition \ref{cor Bruhat B} and the facts that $\left(S \setminus [|i|,|j|-1]\right)\cap F(\sigma,\tau)= S(u,v) \subseteq T(u,v) \subseteq T$ and $(j,\sigma,S \setminus [|i|,|j|-1]) \leqslant (j,\tau, T)$ if and only if $(\sigma,S \setminus [|i|,|j|-1]) \leqslant (\tau,T)$.
 If $j<0<i$ and $|j|<i$
we have the following chain in $(\W(B_{2m+1}),\leqslant)$:
\begin{eqnarray*}
  (i,\sigma, S) &\leqslant& (i,\sigma, S\cup [|j|,i-1]) \\
  &\vartriangleleft& (i-1,\sigma, (S\cup [|j|,i-2])\setminus \{i-1\}) \\
  &\vartriangleleft& (i-2,\sigma, (S\cup [|j|,i-3])\setminus \{i-2,i-1\}) \\
&\vartriangleleft ...\vartriangleleft& (-j,\sigma, S \setminus [|j|,i-1])
\leqslant (j,\sigma,  S \setminus [|j|,i-1]) \leqslant (j,\tau,T),
 \end{eqnarray*} where the last inequality follows as in the previous case.
 If  $j<0<i$ and $|j|>i$ have that $(i,\sigma, S) \leqslant (-i,\sigma,S) \leqslant (j,\tau,T)$, where the second inequality follows by the first case above. The case $i>j>0$ is similar and easier, so we omit it.
 The length formula follows from Proposition \ref{Bposets-odd}.
\end{proof}

We illustrate the previous theorem with an example. Let $m=4$,
$u=[3,4,-5,-6,1,2,9,-7,-8]$ and $v=[-3,-4,-9,1,2,-5,-6,-8,-7]$.
Then $u=(4,[2,-3,1,-4], \{ 2,4 \})$, and $v=(-2,[-2,1,-3,-4],$
$\{ 1,3 \})$,
so $\sigma= [2,-3,1,-4] \leq [-2,1,-3,-4] = \tau$, $F(\sigma, \tau)=\{ 4 \}$, $j=-2 \leq 4=i$, and $S(u,v)= \{ 2,4 \} \cap \{ 1,4 \} \cap \{ 4 \} = \{ 4 \}$, $T(u,v)= \{ 1,3 \} \cap \{  1,4 \} \cap \{ 4 \}= \emptyset$,
so $u \not \leq v$ in $\W(B_9)$.

The following lemma is the analogue of Lemma \ref{lemma unione} and
can be easily deduced from Theorem \ref{Bbruhat-odd} so we omit its verification.
\begin{lem} \label{Blemma unione}
Let $m>0$ and $u,v\in \W(B_{2m+1})$. If $u\leqslant (i,\sigma, S_1) \leqslant v$ and $u\leqslant (i,\sigma, S_2) \leqslant v$ in $(\W(B_{2m+1}),\leqslant)$ then
$u\leqslant (i,\sigma, S_1 \cup S_2) \leqslant v$.
\end{lem}

The next result gives an explicit expression for the M\"obius function of lower intervals in the poset of signed Wachs permutations partially ordered by Bruhat order, and shows,
in particular, that it has values in $\{ 0,1,-1 \}$. The proof is similar to the one of Proposition \ref{moebius-odd} and we omit it.
\begin{pro} \label{moebius-B}
  Let $m>0$, and $v=(j,\tau,T) \in \W(B_{2m+1})$.
Then
$$ \mu(e,v)=\left\{
              \begin{array}{ll}
                (-1)^{|T|}, & \hbox{if $\tau =e$ and $j=m+1$;} \\
                0, & \hbox{otherwise.}
              \end{array}
            \right.
$$
In particular, if $v=(\tau,T) \in \W(B_{2m})$ then
$$ \mu(e,v)=\left\{
              \begin{array}{ll}
                (-1)^{|T|}, & \hbox{if $\tau =e$;} \\
                0, & \hbox{otherwise.}
              \end{array}
            \right.
$$
\end{pro}

By Proposition \ref{moebius-B} and Theorem
\ref{mainB} we deduce the following result.
\begin{cor}
\label{charpolyB}
  The characteristic polynomial of $\W(B_n)$ with the Bruhat order  is
$$
(x-1)^{\left\lfloor\frac{n}{2}\right\rfloor}x^{n^2-\left\lfloor\frac{n}{2}\right\rfloor^2-\left\lfloor\frac{n}{2}\right\rfloor},
$$
for all $n\in \mathbb{P}$.
\end{cor}

\begin{figure} \begin{center}\begin{tikzpicture}

\matrix (a) [matrix of math nodes, column sep=0.3cm, row sep=0.6cm]{
                         &                      & \displaystyle[-1,-2,-3] &                         &        \\
                         &                      & \displaystyle[-2,-1,-3] &                         &        \\
                         &\displaystyle[2,1,-3] &                         & \displaystyle[-3,-1,-2] &        \\
   \displaystyle[1,2,-3] &                      & \displaystyle[-3,-2,-1] &                         &  \displaystyle[3,-1,-2]      \\
                         & \displaystyle[-3,2,1]&                         & \displaystyle[3,-2,-1]  &        \\
   \displaystyle[-3,1,2] &                      & \displaystyle[3,2,1]    &                         &  \displaystyle[-1,-2,3]      \\
                         & \displaystyle[3,1,2] &                         & \displaystyle[-2,-1,3]  &        \\
                         &                      & \displaystyle[2,1,3]    &                         &        \\
                         &                      & \displaystyle[1,2,3]    &                         &        \\};
\foreach \i/\j in {1-3/2-3, 2-3/3-2, 2-3/3-4, 3-2/4-1, 3-4/4-3, 3-4/4-5, 4-1/5-2, 4-3/5-2, 4-3/5-4, 4-5/5-4,%
5-2/6-1,5-2/6-3, 5-4/6-3, 5-4/6-5, 6-1/7-2, 6-3/7-2, 6-5/7-4, 7-2/8-3, 7-4/8-3, 8-3/9-3}
    \draw (a-\i) -- (a-\j);
\end{tikzpicture} \caption{Hasse diagram of $(\W(B_3),\leqslant)$.} \label{fig-D5} \end{center} \end{figure}

\begin{figure}\begin{center} \begin{tikzpicture}

\matrix (a) [matrix of math nodes, column sep=0.05cm, row sep=0.5cm]{
&  &   &33 &  & 32&       & 31    &       & 30    &       &  & \\
&  & 34  &   &16&   & 15    &       & 14    &       & 29    &  & \\
& 35 &   &17 &  & 5 &       & 4     &       & 13    &       &28& \\
\textbf{36}&  &\textbf{18} &   & \textbf{6} &       &       &       & \textbf{3}     &       & \textbf{12}    &  & \textbf{27} \\
&19&   &7  &  & \textbf{1}     &       & \textbf{2}     &       & 11    &       &26& \\
&  &20 &   & \textbf{8}&       & 9     &       & \textbf{10}    &       & 25    &  & \\
&  &   &\textbf{21} &  & 22    &       & 23    &       & \textbf{24}   &       &  & \\};

\foreach \i/\j in {5-6/5-8, 5-8/4-9, 4-9/3-8,3-8/3-6, 3-6/4-5, 4-5/5-4,%
5-4/6-5, 6-5/6-7, 6-7/6-9, 6-9/5-10, 5-10/4-11, 4-11/3-10, 3-10/2-9, 2-9/2-7, 2-7/2-5, 2-5/3-4,%
3-4/4-3, 4-3/5-2, 5-2/6-3, 6-3/7-4, 7-4/7-6, 7-6/7-8, 7-8/7-10, 7-10/6-11, 6-11/5-12, 5-12/4-13,%
4-13/3-12, 3-12/2-11, 2-11/1-10, 1-10/1-8, 1-8/1-6, 1-6/1-4, 1-4/2-3, 2-3/3-2, 3-2/4-1}
\draw (a-\i) -- (a-\j);

\foreach \i/\j in {4-1/4-3, 4-3/4-5, 4-9/4-11,4-11/4-13, 5-6/6-5, 6-5/7-4, 5-8/6-9, 6-9/7-10}
\draw[dashed] (a-\i)-- (a-\j);
\end{tikzpicture}\caption{Reading clockwise, the bold numbers on the diagonals of the hexagon are the sequences $\{3, 12, 27, \ldots\}=\{\rk(\W(B_{2m}))\}_{m>0}$, $\{2\rk(\W(S_{2m}))\}_{m>0}$, $\{\rk(\W(B_{2m+1}))\}_{m \geqslant 0}$
and $\{2\rk(\W(S_{2m+1}))\}_{m>0}$.} \end{center} \end{figure}

\section{Weak orders on Wachs permutations and signed Wachs permutations}

In the previous sections we have proved several results concerning the Bruhat order on Wachs permutations and signed Wachs permutations.
The Bruhat order of a Coxeter group is a refinement of two fundamental orders, the \emph{left weak order} $\leqslant_L$ and the right one $\leqslant_R$  (see, e.g. \cite[Chapter 3]{BB}), whose Hasse diagrams are isomorphic to the Cayley graph of the group, relative to the considered Coxeter presentation. Then it is natural to ask
when our results hold, and in what terms, for the left and right weak orders on Wachs permutations and signed Wachs permutations.

Differently from the Bruhat order, for the right weak order the answer is easily described as a Cartesian product (compare with Propositions \ref{Bposets} and \ref{Bposets-odd}). For reasons that will become clear
in the next two pages we begin with signed Wachs permutations.
Recall that the set of reflections of $B_n$ is $\{(a,b)(-a,-b):1\leqslant a < |b|\leqslant n\} \cup \{(a,-a): 1 \leqslant a \leqslant n\}$. Therefore,
if $v\in B_n$,
$(a,b)(-a,-b) \in T_L(v)$ if and only if $b>0$ and $b$ is to the left of $a$
in the complete notation of $v$, or $b<0$ and $b$ is to the right of $a$, while $(a,-a)\in T_L(v)$ if and only if $a$ is to the left of $-a$.

\begin{thm}
\label{weak-B}
Let $n>0$; then $(\W(B_n),\leqslant_R) \simeq \left(B_{\left \lceil \frac{n}{2} \right\rceil },\leqslant_R\right) \times \mathcal{P}\left(\left[\left\lfloor \frac{n}{2} \right\rfloor \right]\right)$. In particular, $(\W(B_n),\leqslant_R)$ is a complemented lattice.
\end{thm}
\begin{proof}
  Let $n$ be even; for $v=(\tau,T)\in \W(B_n)$ we have that
  \begin{eqnarray*}
    T_L(\tau,T) &=& \biguplus\limits_{(a,-a)_B\in T_L(\tau)}\left\{(2a-1,-2a+1)_B, (2a,-2a)_B, (2a-1,-2a)_B \right\} \\
    &&\biguplus\limits_{(a,b)_B\in T_L(\tau):b>0} \{(2a-1,2b-1)_B,(2a-1,2b)_B,(2a,2b-1)_B,(2a,2b)_B\} \\
    && \biguplus\limits_{(a,b)_B\in T_L(\tau):b<-a} \{(2a-1,2b+1)_B,(2a-1,2b)_B,(2a,2b+1)_B,(2a,2b)_B\} \\
    && \uplus \left\{(v(2a),v(2a-1))_B: a \in T, \, \tau(a)>0 \right\} \\
    &&  \uplus \left\{(-v(2a-1),-v(2a))_B: a \in T, \, \tau(a)<0 \right\}.
  \end{eqnarray*}
  Note that in the first group of reflections the non-symmetric ones
  (i.e., not of the form $(k,-k)$ for some $k \in [n]$)
  are never simple and have odd first element, and even and negative second one. The only simple reflections in the second line have
  even first element and odd positive second one. The ones in the third line are never simple, have negative second element, and
  if the first one is odd and the second one even then the difference
  between the absolute values of the two is $\geq 3$ (and hence are disjoint from the non-symmetric ones in the first line). The reflections in the last two lines are always simple, and have odd first element and even and positive second one, since
  $(v(2a),v(2a-1))_B=(2 \tau(a)-1, 2 \tau(a))_B$ if $\tau(a)>0$ while
  $ (-v(2a-1),-v(2a))_B=(-2 \tau(a)-1, -2 \tau(a))_B$ if $\tau(a)<0$.

  Therefore, if $u=(\sigma,S)\in \W(B_n)$ and $v=(\tau,T)\in \W(B_n)$, we have that $T_L(u) \subseteq T_L(v)$ if and only if $T_L(\sigma) \subseteq T_L(\tau)$ and $\sigma_S \subseteq \tau_T$,
  where, for $X \subseteq [n/2]$ and $w \in B_{n/2}$, $w_X:=\left\{w(i): i\in X\right\}$.
  Indeed, if $T_L(u) \subseteq T_L(v)$ and $(a,b)_B \in T_L(\sigma)$ is such that (say) $b<-a$
  then, by the equation written at the beginning of this proof, $(2a,2b)_B \in T_L(u)$ so $(2a,2b)_B \in T_L(v)$ which implies,
  by the remarks above,
  that there is $(c,d)_B \in T_L(\tau)$ such that $c<-d$ and  $(2a,2b)_B=(2c,2d)_B$, so $(a,b)_B =(c,d)_B \in T_L(\tau)$. Similarly
  all the other cases. The converse is clear.
  Hence, by Proposition \ref{weak-order}, the map $(\tau,T) \mapsto (\tau,\tau_T)$ gives the desired poset isomorphism.

  Let $n$ be odd and $v=(j,\tau,T) \in \W(B_n)$. Then similarly
  \begin{eqnarray*}
    T_L(j,\tau,T) &=& \biguplus\limits_{(a,-a)_B\in T_L(\bar{\tau}): a\neq \lceil n/2 \rceil}\left\{(2a-1,-2a+1)_B, (2a,-2a)_B, (2a-1,-2a)_B \right\} \\
    &&\biguplus\limits_{(a,b)_B\in T_L(\bar{\tau}):0<b< \lceil n/2 \rceil} \{(2a-1,2b-1)_B,(2a-1,2b)_B,(2a,2b-1)_B,(2a,2b)_B\} \\
    && \biguplus\limits_{(a,b)_B\in T_L(\bar{\tau}):-\lceil n/2 \rceil<b<-a} \{(2a-1,2b+1)_B,(2a-1,2b)_B,(2a,2b+1)_B,(2a,2b)_B\} \\
    && \biguplus\limits_{(a,b)_B\in T_L(\bar{\tau}):b= \lceil n/2 \rceil} \{(2a-1,n)_B, (2a,n)_B\} \\
    && \biguplus\limits_{(a,b)_B\in T_L(\bar{\tau}):b=- \lceil n/2 \rceil, a\neq -b} \{(2a-1,-n)_B, (2a,-n)_B\} \\
    && \uplus  \{(n,-n)_B: (\lceil n/2 \rceil,-\lceil n/2 \rceil)_B\in T_L(\bar{\tau})\} \\
    && \uplus \left\{(2 \tau(a)-1,2\tau(a))_B: a \in T, \, \tau(a)>0 \right\}\\
    &&  \uplus \left\{(2 \tau(a)-1,2 \tau(a))_B: -a \in T, \, \tau(a)>0 \right\} \\
  \end{eqnarray*} where $$\bar{\tau}(k):= \left\{
                                                 \begin{array}{ll}
                                                   \tau(k), & \hbox{if $k<|j|$;} \\
                                                   \sgn(j)\cdot(n+1)/2, & \hbox{if $k=|j|$;} \\
                                                   \tau(k-1), & \hbox{if $k>|j|$}, \\
                                                 \end{array}
                                               \right.$$
for all $k\in \lceil n/2 \rceil$. The considerations made in the
even case apply line by line also to this case, with the added
remark that the reflections in lines 4, 5, and 6 are the only ones
with $n$ or $-n$ in the second position.
Therefore, if $u=(i,\sigma,S)\in \W(B_n)$ and $v=(j,\tau,T)\in \W(B_n)$,
we have that $T_L(u) \subseteq T_L(v)$ if and only if $T_L(\bar{\sigma}) \subseteq T_L(\bar{\tau})$ and $\sigma_S\subseteq \tau_S$, and we conclude as in the previous case. The last statement follows from the
fact that $(B_{\left \lceil \frac{n}{2} \right\rceil },\leqslant_R)$ is a complemented lattice (see, e.g., \cite[Cor. 3.2.2]{BB}).
\end{proof}
Regarding left weak order we have the following observations.
\begin{pro} Let $m>0$. The posets $(\W(B_{2m}),\leqslant_R)$ and $(\W(B_{2m}),\leqslant_L)$ are isomorphic.
\end{pro}
\begin{proof}
  The assignment $v \mapsto v^{-1}$ defines a bijective function $\W(B_{2m}) \rightarrow \W(B_{2m})$ and $u \leqslant_R v$ if and ony if $u^{-1} \leqslant_L v^{-1}$, for all $u,v\in B_{2m}$.
\end{proof} \noindent On the other hand, the poset $(\W(B_3),\leqslant_L)$ is not graded.

For Wachs permutations the situation is analogous but simpler so we leave to the interested reader the proof of the following result.
\begin{thm} \label{weak-A}
Let $n>0$; then $(\W(S_n),\leqslant_R) \simeq \left(S_{\left \lceil \frac{n}{2} \right\rceil },\leqslant_R\right) \times \mathcal{P}\left(\left[\left\lfloor \frac{n}{2} \right\rfloor \right]\right)$. In particular, $(\W(S_n),\leqslant_R)$ is a
complemented lattice.
\end{thm} As in the case of signed Wachs permutations, the posets $(\W(S_{2m}),\leqslant_R)$ and $(\W(S_{2m}),\leqslant_L)$ are isomorphic, for all $m>0$. On the other hand, the poset $(\W(S_5),\leqslant_L)$ is not graded.

\section{Open problems}

In this section we collect some open problems and conjectures which arise in this work, and the evidence that we have in their favor.

We have proved in Propositions \ref{moebius-odd} and \ref{moebius-B}
that the M\"{o}bius function of lower intervals in the posets of Wachs permutations and signed Wachs permutations partially ordered by Bruhat order always
has values in $\{ 0, 1, -1 \}$. We feel that this is true in general.

\begin{con}
\label{mobiusA}
Let $n \in \PP$. Then
\[
\mu(u,v) \in \{ 0, 1, -1 \}
\]
for all $u,v \in \W(S_n)$.
\end{con}
 We have verified Conjecture \ref{mobiusA} for $n \leq 8$.
\begin{con}
\label{mobiusB}
Let $n \in \PP$. Then
\[
\mu(u,v) \in \{ 0, 1, -1 \}
\]
for all $u,v \in \W(B_n)$.
\end{con}
We have verified Conjecture \ref{mobiusB} for $n \leq 6$.
Note that since $\W(S_n)$ is isomorphic, as a poset, to the
interval $[e,[n, \ldots ,3,2,1]]$ in $\W(B_n)$, Conjecture
\ref{mobiusB} implies Conjecture \ref{mobiusA}.

Recently Davis and Sagan \cite{sagan} studied the convex hull of various sets
of pattern avoiding permutations. Following this idea, it is natural
to look at the convex hulls $c(\W(S_n))$ and $c(\W(B_n))$ of Wachs
and signed Wachs permutations in $\RR^n$. In this respect, we feel that
the following is true.
\begin{con}
\label{simpleA}
Let $m \in \PP$. Then $c(\W(S_{2m}))$ is a simple polytope.
\end{con}
We have verified Conjecture \ref{simpleA} for $m \leq 5$.
According to SageMath \cite{sagemath} $c(\W(S_9))$ is not simple.

\begin{con}
\label{simpleB}
Let $m \in \PP$. Then $c(\W(B_{2m}))$ is a simple polytope.
\end{con}
We have verified Conjecture \ref{simpleB} for $m \leq 3$.
According to SageMath \cite{sagemath} $c(\W(B_3))$ is not simple.

Regarding left weak order, we feel that the following might be true.
\begin{prob}
\label{latticeAodd}
Is $(\W(S_{2m+1}),\leqslant_L)$ a lattice
for all $m \in \PP$?
\end{prob}
We have verified that the answer to Problem \ref{latticeAodd} is
yes if $m \leq 4$.

\noindent
{\bf Acknowledgments.}
The first named author was partially supported by the MIUR Excellence Department Project
CUP E83C18000100006.

\end{document}